\documentclass[ejsv2]{imsart}

\RequirePackage[numbers]{natbib}
\RequirePackage[colorlinks,citecolor=blue,urlcolor=blue]{hyperref}
\RequirePackage{graphicx}
\usepackage{hyperref}
\arxiv{2411.13199}
\startlocaldefs
\theoremstyle{plain}

\newtheorem{theorem}{Theorem}[section]
\newtheorem{lemma}[theorem]{Lemma}
\theoremstyle{definition}

\newtheorem*{remark}{Remark}
\newtheorem{assumption}{Assumption}
\theoremstyle{remark}

\DeclareMathOperator*{\argmin}{argmin}
\DeclareMathOperator{\var}{Var}
\DeclareMathOperator{\cov}{Cov}
\newcommand{\sumn}{\sum_{i=1}^{n}}
\newcommand{\sumij}{\sum_{i=1}^{m_1}\sum_{j=1}^{m_2}}
\newcommand{\dif}{\hat{A}_H-A_{0}}

\DeclareMathOperator{\sign}{sign}

\newcommand{\V}{\Vert}
\newcommand{\Vf}{\Vert_{w(F)}}
 \newcommand{\ind}{\textbf{1}}

\newcommand{\sxix}{\frac{1}{n}\sum_{i=1}^n \xi_iX_i}%sum xi X
\newcommand{\spx}{\frac{1}{n}\sum_{i=1}^n \phi_{\tau}(\xi_i)X_i}
\newcommand{\ra}{\rangle}
\newcommand{\la}{\langle}
\newcommand{\sumi}{\frac{1}{n}\sum_{i=1}^n}
\newcommand{\Ah}{\hat{A}}
\newcommand{\At}{\tilde{A}}
\newcommand{\Ahh}{\hat{A}_H}
\newcommand{\As}{\hat{A}_S}
\newcommand{\Eb}{\mathbb{E}}
\newcommand{\Ec}{\mathcal{E}_i(j,k)}
\newcommand{\Pb}{\mathbb{P}}
\newcommand{\Rb}{\mathbb{R}}

\newcommand{\Rc}{\mathcal{R}}

\newcommand{\xb}{\bar{\xi}}
\newcommand{\lb}{\left(}
\newcommand{\rb}{\right)}
\newcommand{\lnorm}{\left \|}
\newcommand{\rnorm}{\right\|}
\newcommand{\lbr}{\left(}
\newcommand{\rbr}{\right)}
\newcommand{\indc}{\textbf{1}_{\chi_i}}
\newcommand{\daswf}{\|\hat{A}_S-A_0\|_{w(F)}}%wf wrighted Frobenius norm
\newcommand{\dasf}{\|\hat{A}_S-A_0\|_{F}}%f Frobenius norm

\endlocaldefs

\begin{document}
\begin{frontmatter}
\title{Sharp Bounds for Multiple Models in Matrix Completion}
%\title{A sample article title with some additional note\thanksref{t1}}
\runtitle{Sharp Bounds for  Matrix Completion}
%\thankstext{T1}{A sample additional note to the title.}

\begin{aug}
%%%%%%%%%%%%%%%%%%%%%%%%%%%%%%%%%%%%%%%%%%%%%%%
%% Only one address is permitted per author. %%
%% Only division, organization and e-mail is %%
%% included in the address.                  %%
%% Additional information can be included in %%
\author[A]{\fnms{Dali}~\snm{Liu}\ead[label=e1]{liudali@msu.edu}},
\author[A]{\fnms{Haolei}~\snm{Weng}\ead[label=e2]{whaoleng@msu.edu}}

%%%%%%%%%%%%%%%%%%%%%%%%%%%%%%%%%%%%%%%%%%%%%%
%% Addresses                                %%
%%%%%%%%%%%%%%%%%%%%%%%%%%%%%%%%%%%%%%%%%%%%%%
\address[A]{Department of Statistics and Probability,
Michigan State University\printead[presep={,\ }]{e1,e2}}

\runauthor{D. Liu and H. Weng}
\end{aug}

\begin{abstract}
In this paper, we demonstrate how a class of advanced matrix concentration inequalities, introduced in  \cite{brailovskaya2024universality}, can be used to eliminate the dimensional factor in the convergence rate of matrix completion.  This dimensional factor represents a significant gap between the upper bound and the minimax lower bound, especially in high dimension. Through a more precise spectral norm analysis, we remove the dimensional factors for three popular matrix completion estimators, thereby establishing their minimax rate optimality. 
\end{abstract}

\begin{keyword}[class=MSC]
\kwd[Primary ]{62J99}
\kwd{62H12}
\kwd[; secondary ]{60B20}
\kwd{15A83}
\end{keyword}

\begin{keyword}
\kwd{Matrix completion}
\kwd{low-rank matrix estimation}
\kwd{minimax optimality}
\end{keyword}

\end{frontmatter}
%%%%%%%%%%%%%%%%%%%%%%%%%%%%%%%%%%%%%%%%%%%%%%
%% Please use \tableofcontents for articles %%
%% with 50 pages and more                   %%
%%%%%%%%%%%%%%%%%%%%%%%%%%%%%%%%%%%%%%%%%%%%%%
%\tableofcontents

\section{Introduction}\label{section: intro}
Matrix recovery from a small subset of entries, often referred to as matrix completion, is a classical problem in high-dimensional statistics with wide-ranging applications across various fields.  Numerous frameworks and methods have been developed for matrix completion, with a fundamental understanding that a large matrix cannot be recovered unless it exhibits certain low-rank properties.  As a convex relaxation of rank penalization,  nuclear norm penalized trace regression has become a popular approach and corresponding theoretical research has been fruitful. See, for example, \cite{candes2010power}, \cite{recht2011simpler}, \cite{koltchinskii2011nuclear}, \cite{negahban2012restricted}, \cite{klopp2014noisy}, \cite{klopp2017robust}, \cite{minsker2018sub}, \cite{yu2024low}.

In particular, many studies  focus on the convergence rate of the nuclear norm penalization methods \cite{koltchinskii2011nuclear, negahban2012restricted, klopp2014noisy, yu2024low}. However, the minimax lower bound, obtained in \cite{koltchinskii2011nuclear} and \cite{negahban2012restricted}, does not include a logarithmic dimension factor that always appears in the upper bound of the convergence rate. 
This discrepancy between the two bounds becomes more pronounced in high-dimensional settings, highlighting a potential gap in the existing theoretical understanding.
Due to the ubiquitousness of this issue, all the previous works have had to qualify their results by stating ``our estimator is minimax optimal, up to a logarithmic factor''. 
In this paper, we leverage a class of sharp matrix concentration inequalities to remove the logarithmic dimension factor.
After our work, the qualification will no longer be necessary.

\subsection*{Problem setup}
Let $A_0\in \Rb^{m_1\times m_2}$ be the unknown matrix we aim to recover. We observe independent samples $\{(X_i, Y_i)\}_{i=1}^{n}$ from the model
\begin{equation}
\label{model:mc}
Y_i=\la X_i, A_0\ra +\xi_i,i=1,\ldots, n,
\end{equation}
where $X_i\in \mathbb{R}^{m_1\times m_2}$, $\langle X_i, A_0\rangle=\text{tr}(X_i^TA_0)$, and $\text{tr}(B)$ denotes the trace of  matrix $B$.
The sampling matrices $X_i\in\mathbb{R}^{m_1\times m_2}$ take values from $\{ e_j(m_1)e_k^T(m_2),\\ 1\leq j\leq m_1, 1\leq k\leq m_2\}$, where 
$\{e_j(m_1)\}_{j=1}^{m_1},\{e_k(m_2)\}_{k=1}^{m_2}$ denote standard basis vectors in $\mathbb{R}^{m_1}, \mathbb{R}^{m_2}$, respectively. This implies that each $X_i$ only has one non-zero entry,  indicating the location of the sampled entry.

\subsection*{The logarithmic dimension factor issue}

The matrix completion problem is often addressed by a penalized regression.  Assuming $A_0$ is of low rank, 
a natural choice for the penalty is the rank of a matrix.
However, the rank penalty is non-convex and makes the corresponding  program difficult to solve. 
The nuclear norm $\|\cdot\|_*$, namely the sum of all the singular values of a matrix, is usually used as a convex relaxation for the rank.
There are  many classical works on nuclear norm penalization methods under different assumptions. 
Assuming the noise has exponential decay,  \cite{koltchinskii2011nuclear} consider  the case that the sampling distribution is known; 
\cite{negahban2012restricted} study the situation where  the row indices and the column indices of the observations are sampled independently and the penalty is a weighted nuclear norm;
\cite{klopp2014noisy} explores the problem for general sampling distributions. Consider the classical nuclear norm penalized least squares estimator
\begin{equation*}
 \hat{A}=\argmin_{\|A\|_{\infty}\leq a }  \left\{ \dfrac{1}{n}\sum_{i=1}^n (Y_i-\langle X_i,A\rangle)^2+\lambda \|A\|_*  \right\},
\end{equation*}
where $\|A\|_{\infty}=\max_{1\leq j\leq m_1, 1\leq k\leq  m_2} |a_{jk}|$ and $a>0$.
Assuming $\text{rank}(A_0)\leq r$, $X_i$ has a  general  distribution and the noise has exponential decay, \cite{klopp2014noisy} obtains the following bound that holds with high probability:
\begin{equation}\label{inequality: convergence rate of replacement model}
 \dfrac{\|\hat A-A_0\|_F^2}{m_1m_2} \lesssim \log (m_1+m_2) \dfrac{r\max(m_1,m_2)}{n},  
\end{equation}
where $\lesssim$ means that the inequality holds up to some numeric constant, and $\|A\|_F=\\\sqrt{\sum_{j=1}^{m_1}\sum_{k=1}^{m_2}a^2_{jk}}$ is the Frobenius norm of $A=(a_{jk})$. However, the minimax lower bound, as established in \cite{koltchinskii2011nuclear}, does not involve the term $\log (m_1+m_2)$ and is expressed as
\begin{equation*}
   \inf_{\hat{A}}\sup_{\mathcal{A}} \Pb\left( \dfrac{\|\hat A-A_0\|_F^2}{m_1m_2} \gtrsim \dfrac{r \max(m_1,m_2)}{n}\right)\geq C,
\end{equation*}
where $C\in (0,1)$ is a constant, and the infimum and supremum are taken over all the possible estimators and the parameter space $\{A_0\in \mathbb{R}^{m_1\times m_2}: \text{rank}(A_0)\leq r, \|A_0\|_{\infty}\leq a\}$ respectively. The logarithmic dimension factor $\log (m_1+m_2)$, as the gap between the upper and lower bounds, exists in other matrix completion settings as well. See, for instance,  \cite{klopp2017robust} for corrupted matrix completion,  \cite{yu2024low} for matrix completion with heavy tailed noise. 

By our analysis, this gap arises essentially from the spectral norm analysis of some random matrices.
In matrix completion, we frequently need to control the spectral norms of random matrices of the form $\frac{1}{n}\sum_{i=1}^n \zeta_i X_i$, where $\zeta_i$ are bounded or light-tailed random variables.
The use of traditional concentration inequalities in this context typically introduces a logarithmic dependence on the dimension.  
This limitation is not unique to matrix completion---it arises in a variety of statistical and mathematical problems and has garnered significant attention in recent years.
Notably, \cite{bandeira2023matrix} and \cite{brailovskaya2024universality} proposed sharp new concentration inequalities that can successfully remove the logarithmic dependence in various settings.
These results will serve as key tools in our paper.

Another point worth mentioning is that there is an alternative sampling model in matrix completion. In this model, there are Bernoulli random variables $\eta_{ij}$ indicating whether the entry of $A_0$ at location $(i,j)$ is observed. Under this assumption, the observations can be written as:
\begin{equation*}
\mathcal{O} = \{a_{0,ij} + \xi_{ij} : \eta_{ij} = 1,\  1\leq i\leq m_1 ,1\leq j \leq  m_2\},
\end{equation*}
where $\xi_{ij}$ represents the noise at location $(i,j)$ and $A_0=(a_{0,ij})$.
This sampling model is often referred to as \emph{sampling without replacement} because each entry can be selected at most once. 
A rich body of literature has investigated this model \citep{candes2010matrix,candes2012exact,keshavan2009matrix,chen2015fast,chen2020noisy}. As a contrast, the sampling strategy considered in our paper is known as \emph{sampling with replacement}, as it allows repeated observations of the same entry.
Since $n$ is often assumed to be much smaller than the matrix dimension $m_1\times m_2$, the number of repeated observations could be very small. Despite different sampling strategies, the theoretical results derived in both models are often compared side by side \citep{chen2015fast,chen2020noisy}. 
For instance, some works derive bounds of the following form \citep{chen2015fast}:
\begin{equation}\label{inequality: convergence rate of bernoulli model}
 \quad \frac{\|\widehat{A} - A_0\|_F^2}{m_1 m_2} \lesssim \frac{r\max(m_1,m_2)}{m_1 m_2 p},
\end{equation}
under the non-replacement model with $\eta_{ij}$ following $  \text{Bernoulli}(p)$. Here, as the expected sample size, $m_1m_2p$ plays a role as $n$ in the with-replacement model.
Compared to (\ref{inequality: convergence rate of bernoulli model}), the extra $\log(m_1+m_2)$ factor in (\ref{inequality: convergence rate of replacement model}) is a common critique of algorithms developed under the with-replacement model \citep{chen2015fast,chen2020noisy}.
Therefore, removing this logarithmic factor can help significantly defend the theoretical validity of the algorithms \citep{koltchinskii2011nuclear,negahban2012restricted,klopp2014noisy}.

\subsection*{Our contribution}
In this paper, we revisit three estimators that have been studied in other research works:
\begin{itemize}
\item The unknown matrix is of low rank and the noise is heavy tailed. 
The estimator in \cite{yu2024low} is designed for the case that the noise is only assumed to have  finite second moment.
\item  The unknown matrix is of low rank and the noise is sub-Gaussian. Nuclear norm penalized least squares is a common approach under this setting \citep{koltchinskii2011nuclear,negahban2012restricted,klopp2014noisy}. We consider the estimator proposed in \cite{klopp2014noisy} to illustrate our new results.
 \item  The unknown matrix is of low rank and the noise is sub-Gaussian with unknown variance. 
 The parameter tuning in matrix completion often requires prior knowledge of the noise variance.
\cite{klopp2014noisy} proposes an estimator to solve such a problem.
\end{itemize}
All the previous results on the convergence rates of the above three estimators involve the logarithmic dimension factor.
We remove the dimensional factor for all the three estimators, thereby proving their minimax rate optimality. 

The remainder of the paper is organized as follows.
In Section \ref{section: main}, we present our main results. 
Section \ref{section: proofs} is dedicated to the proofs.
We give some concluding remarks in Section \ref{section: discussion}. 

\section{Main Results}\label{section: main}
In this section, we present new upper bounds for the three estimators mentioned in Section \ref{section: intro}. 
In particular, compared to the existing bounds, the new bounds do not involve the dimension factor $\log d$ and hence match with the minimax lower bounds. 

 We first collect some  notations that are used throughout the paper. We denote $M=\max(m_1,m_2), m=\min(m_1,m_2), d= m_1+m_2$. We write $(\log d)^\alpha$ as $\log^{\alpha} d$  for aesthetic reasons. $C, C_1, C_2,\ldots$ are positive constants whose value may vary at each occurrence.

Recall our observations are independent noised entries sampled from $A_0\in \Rb^{m_1\times m_2}$:
\begin{equation*}
    Y_i=\la X_i, A_0\ra +\xi_i, \quad i=1,\ldots, n.
\end{equation*}
The sampling distribution of $X_i$ is denoted by 
\begin{equation*}
    P_{jk}=\Pb(X_i=e_j(m_1)e_k^T(m_2)), \quad 1\leq j\leq m_1, 1\leq k\leq m_2.
\end{equation*}
Throughout the paper,  we assume 
\begin{align*}
    \text{rank}(A_0)\leq r,\\
    \|A_0\|_{\infty}\leq  a.
\end{align*}

Regarding the distribution of $X_i$, we consider a general sampling distribution.
\begin{assumption}\label{assumption:X_i}
$X_i$'s distribution satisfies
\begin{align*}
 (1)~&\max_j \sum_{k=1}^{m_2} P_{jk}\leq \frac{L_2}{m},\\
 (2)~&\max_k \sum_{j=1}^{m_1} P_{jk}\leq \frac{L_2}{m},\\
 (3)~&\min_{j,k} P_{jk}\geq \frac{1}{\mu m_1m_2},\\
  (4)~&\max_{j,k} P_{jk}\leq \frac{L_3}{m\log^3 d},
  \end{align*}
where  $L_2, L_3\geq 0$ are constants, and $\mu\geq 1$ is allowed to change with the sample size or dimension. 
\end{assumption}

Conditions $(1)$-$(3)$ of Assumption \ref{assumption:X_i} are the same as in \cite{klopp2014noisy}. $L_2$ is introduced to control the row and column sampling probabilities, preventing them from being too large. Condition (4) requires that the sampling doe not extremely concentrate on a few entries. Note that Condition (1) or (2) already implies $\max_{j,k}P_{jk}\lesssim \frac{1}{m}$ which only differs from Condition (4) by a logarithmic factor. Hence, (4) is a very mild condition.

In following subsections, we introduce our $\log d$-free error bounds.

\subsection{Matrix completion with heavy tailed noise}\label{subsec:hv}
We start with our improvement on a recent advance in matrix completion \citep{yu2024low}.
Heavy-tailed data is ubiquitous in many fields such as finance, microeconomics and biology. To address heavy-tailed noise in the context of matrix completion, several estimators have been proposed. 
Assuming that  the second moment of the noise is finite, \cite{minsker2018sub} proposes a two-step procedure, where a truncated-type estimator $\tilde A$ is obtained first and then a nuclear norm penalized regression is performed on $\tilde A$.
\cite{fan2021shrinkage} also proposes a two-step procedure while a stronger moment condition $\Eb[|\xi_i|^{k}]<\infty$ for $k>2$ is required.
\cite{elsener2018robust} considers a class of loss functions including Huber loss function. 
They assume the distribution of the noise is symmetric and some additional conditions hold.
In particular, \cite{yu2024low} only assumes  that the noise variables have finite second moments and still derives strong properties for the estimator under study. The estimator proposed in their paper is based on the Huber loss function, which is one of the fundamental tools for addressing  heavy tailed noise.

We now formally state our assumption on the distribution of $\{\xi_i\}_{i=1}^n$. We only assume the second moment is finite.
\begin{assumption}\label{assumption: mean var}
Given $X_i$, the noise variables satisfy
\begin{equation*}
\Eb[\xi_i|X_i]=0    
\end{equation*}
and their second moments are uniformly bounded
\begin{equation*}
 \Eb[\xi_i^2|X_i]\leq \sigma^2.   
\end{equation*}
\end{assumption}

To deal with the heavy tailed noise, we introduce the following function 
\begin{align*}
l_{\tau}(x)=\begin{cases}
\frac{x^2}{2}, ~~|x|\leq \tau,\\
\tau|x|-\frac{\tau^2}{2} , ~~|x|>\tau,\\
\end{cases}
\end{align*}
and Huber loss is defined as $\frac{1}{n}\sum_{i=1}^n l_{\tau}(Y_i-\langle X_i, A\rangle)$. 
The estimator to be studied is
\begin{align}\label{estimator: heavy tail}
\hat{A}_H=\argmin_{\Vert A\Vert_{\infty}\leq a} \left\{ \dfrac{1}{n}\sumn l_{\tau}(Y_i-\la X_i, A\ra)+\lambda\Vert A\Vert_*\right\}.
\end{align}
We obtain the following new convergence rate for $\hat{A}_H$.

\begin{theorem}\label{theorem: new heavy tail}

Recall our basic  assumption and notation: $\|A_0\|_\infty\leq a$, $\text{rank}(A_0)\leq r$, $d=m_1+m_2$, $m=\min(m_1,m_2)$ and $M=\max(m_1,m_2)$, where $A_0\in \mathbb{R}^{m_1\times m_2}$.\\

\textbf{Part I}:
Let Assumptions \ref{assumption:X_i}  and  \ref{assumption: mean var} hold. Suppose  $n\geq 160m\log^4 d$ and $\xi_i|X_i$ are symmetric.  There exist constants  $C_1$\footnote{The constant in the tuning parameter $\lambda$ can vary within a certain range. This is true for other theorems in the paper as well.} and $ C_2$ such that 
with 
 \begin{align*}
  &\lambda=C_1\max(\sigma,a)\sqrt{\frac{1}{nm}},\\
  &\tau= \frac{\max(\sigma,a)}{\log^2 d}\sqrt{\frac{n}{m}},
 \end{align*}
 the following bound holds  with probability at least $1-\frac{2}{d}$,
\begin{equation}\label{inequality: new heavy tail}
\dfrac{\|\hat{A}_H-A_0\|^2_F}{m_1m_2}  \leq C_2\frac{\mu^2\max(a^2,\sigma^2)rM}{n}.  
\end{equation}
Here, $C_1$ and $C_2$ depend on the constants $L_2, L_3$ from Assumption \ref{assumption:X_i}.

\textbf{Part II}: Let Assumption \ref{assumption: mean var} hold  and the sampling be uniform,  $P_{ij}=\frac{1}{m_1m_2}$. 
Suppose   $n\geq 160m\log^5 d$ and $M\geq m\log^4d$.  There exist universal constants $C_1$ and $ C_2$ such that with the tuning parameters 
 \begin{align*}
  &\lambda=C_1\max(\sigma,a)\sqrt{\frac{1}{nm}},\\
  &\tau= \frac{\max(\sigma,a)}{\log^2 d}\sqrt{\frac{n}{m}},
 \end{align*}
the following bound holds with probability at least $1-\frac{3}{d}$,
\begin{equation*}
\dfrac{\|\hat{A}_H-A_0\|^2_F}{m_1m_2}  \leq C_2\frac{\mu^2\max(a^2,\sigma^2)rM}{n}.  
\end{equation*}

\end{theorem}

\cite{yu2024low} provides the following theorem.
\begin{theorem}\label{theorem: heavy tail}
Let Assumption  \ref{assumption: mean var} hold, and suppose  $P_{jk}=\frac{1}{m_1m_2}$.
Assume \(n \ge C_1\, m \log d\) for some universal constant $C_1$. 
There exist universal constants $\{C_i\}_{i=2}^4$ such that 
with 
\begin{align*}
    &\lambda=C_2\max(\sigma, a)\sqrt{\frac{\log d}{nm}},\\
    &\tau=C_3\max(\sigma, a) \sqrt{\frac{n}{m\log d}},
\end{align*}
the following holds with probability at least $1-\frac{3}{d}$,
\begin{equation*}
\dfrac{\|\hat{A}_H-A_0\|_F^2}{m_1m_2}\leq C_4  \max\left\{\max(a^2,\sigma^2) \dfrac{rM\log d}{n},  a^2 \sqrt{\frac{\log d}{n}}\right\}.   
\end{equation*}
\end{theorem}%End of Theorem Old heavy tailed

Below we provide a comparison between Theorem~2.1 and Theorem~2.2. 
\begin{enumerate}
\item Our new result eliminates the $\log d$ factor in the convergence rate. It is achieved by utilizing new advanced concentration inequalities in \cite{brailovskaya2024universality}. See our Lemma \ref{lemma: sharp spectral} for a detailed analysis. \cite{koltchinskii2011nuclear} provides the following minimax lower bound in the case of gaussian noise and uniform sampling distribution:
\begin{equation}\label{inequality: minimax lower}
    \inf_{\hat{A}}\sup_{A_0} \Pb\left(\dfrac{\|\hat{A}-A_0\|_F^2}{m_1m_2}\geq C_1  \dfrac{\min(\sigma^2,a^2) rM}{n}\right)\geq C_2,
\end{equation}
where $C_1,C_2$ are universal constants, and the infimum and supremum are taken over all possible estimators and rank-$r$ matrices respectively. Combining the new bound \eqref{inequality: new heavy tail} with the minimax lower bound \eqref{inequality: minimax lower} shows that the removal of the dimension factor $\log d$ establishes the minimax rate optimality of the estimator \eqref{estimator: heavy tail}.

\item Theorem \ref{theorem: new heavy tail} allows for a more general sampling distribution while Theorem \ref{theorem: heavy tail} requires uniform distribution.
We also eliminate  the nuisance term $O(\sqrt{(\log d)/n})$ from the convergence rate, via a different peeling argument \citep{wainwright2019high} in the analysis of certain empirical processes.

\item Our analysis reveals a more precise choice for the tuning parameter $\lambda$: the optimal choice for $\lambda$ should be of the order $O(\sqrt{1/(nm)})$ instead of $O(\sqrt{(\log d)/(nm)})$, which sheds light on the practice of parameter tuning.

\item 

Theorem \ref{theorem: new heavy tail} requires   $n\geq Cm\log^4 d$ (Part I) or $n\geq Cm\log^5 d$ (Part II) to remove the $\log d$ factor, while Theorem \ref{theorem: heavy tail} requires $n\geq Cm\log d$. This slightly stronger condition results from leveraging the sharp concentration inequalities in \cite{brailovskaya2024universality}. 
We emphasize that by the previous techniques, the $\log d$ free bounds cannot  be obtained even if one assumes $n\geq Cm \log^4 d$.
We defer a more detailed discussion on this issue to the remarks after Theorem \ref{theorem: klopp lasso}. 

\item

In our analysis, we consider some random matrices constructed from truncated
versions of $\{\xi_i\}_{i=1}^n$. When the noise follows a general
(not necessarily symmetric) distribution, the random matrices may not be mean zero and additional bias terms are introduced. 
Controlling the bias terms requires the conditions that the sampling is uniform and $M\geq m\log^4d$.

Moreover, assuming that  $\xi_i$ have  finite moments of order $2 + \kappa$ for some $\kappa > 0$, a $\log d$-free bound can be established 
without restrictions on the matrix dimensions. In this case, the following sample size condition is required
\[
n \geq 
C m\log^{4+\frac{4}{\kappa}} d.
\]
See our analysis in the proof of Theorem \ref{theorem: new heavy tail} Part II for details.
\end{enumerate}

\subsection{Matrix completion with known variance}\label{subsec:known}
The setting that the noise have distributions with exponential decay is most studied, see, e.g., \cite{koltchinskii2011nuclear, negahban2012restricted, klopp2014noisy}. In this section, we assume the noise $\{\xi_i\}_{i=1}^n$ are sub-Gaussian.

\begin{assumption}\label{assumption:noise-subgaussian}
The noise variables $\xi_i$ satisfy $\Eb[\xi_i|X_i]=0$. Its overall variance satisfies $\Eb[\xi^2_i]=\sigma^2$ and given $X_i$, it holds uniformly
\begin{equation*}
 \Eb[\exp(t\xi_i/\sigma)|X_i]\leq \exp(L_1^2t^2), \forall t\in\mathbb{R},
\end{equation*}
where $L_1> 0$ is a constant.
\end{assumption}

Under the assumption of  light tailed noise,  an estimator proposed in \cite{klopp2014noisy}  is:
\begin{equation}\label{estimator:lasso}
\hat{A}=\argmin_{\|A\|_\infty\leq a}  \left\{ \dfrac{1}{n}\sum_{i=1}^n (Y_i-\langle X_i,A\rangle)^2+\lambda \|A\|_*  \right\}.
\end{equation}
We now present our new result:
\begin{theorem}\label{theorem: new lasso}
 Let Assumptions   \ref{assumption:X_i} and \ref{assumption:noise-subgaussian} hold. 
 Assume that
 $n\geq C_1m\log^4 d~(\log d+\log n)$.  There exist $C_2$ and $C_3$ such that 
 with 
 \begin{equation*}
\lambda=C_2\sigma\sqrt{\frac{1}{nm}} ,
 \end{equation*}
the following bound holds with probability at least $1-\frac{2}{d}$
\begin{equation}\label{inequality: new lasso}
  \dfrac{\|\hat{A}-A_0\|_F^2}{m_1m_2} \leq  \dfrac{C_3\mu^2\max(a^2,\sigma^2)rM}{n}.  
\end{equation}
Here, $C_1$ is a universal constant, and $C_2, C_3$ depend on the constants $\{L_i\}_{i=1}^3$ from Assumptions \ref{assumption:X_i} and \ref{assumption:noise-subgaussian}.
\end{theorem}
To compare with existing bounds, we mention the following theorem.
\begin{theorem}[Theorem 7 in \cite{klopp2014noisy}]\label{theorem: klopp lasso}
Let Assumption \ref{assumption:noise-subgaussian},  Conditions (1)-(3) in  Assumption \ref{assumption:X_i}  hold. 
Assume that $n\geq C_1m\log^2 d$. 
There exist $\{C_i\}_{i=1}^3$ such that
with 
\begin{equation*} \lambda=C_2\sigma\sqrt{\dfrac{\log d}{nm}}, 
\end{equation*}
the following bound holds with probability $1-\frac{3}{d}$:
\begin{equation*}
 \dfrac{\|\hat{A}-A_0\|_F^2}{m_1m_2}\leq   C_3\max \left\{\dfrac{\mu^2\max(a^2,\sigma^2)rM\log d}{n},a^2\mu\sqrt{\frac{\log d}{n}} \right\}. 
\end{equation*}
Here, $C_1$ is a universal constant, and $C_2, C_3$ depend on the constants $L_1, L_2$ from Assumptions \ref{assumption:X_i} and \ref{assumption:noise-subgaussian}.
\end{theorem}

Comparing Theorems \ref{theorem: new lasso} and \ref{theorem: klopp lasso}, we make several remarks:
\begin{enumerate}
 \item  Theorem \ref{theorem: new lasso} eliminates a $\log d$ factor in the convergence rate.
 Comparing (\ref{inequality: new lasso}) and (\ref{inequality: minimax lower}), the removal of the $\log d$ factor establishes the minimax optimality of estimator (\ref{estimator:lasso}). 
 
  \item Theorem \ref{theorem: new lasso} also provides the correct order of tuning parameter $\lambda$ for (\ref{estimator:lasso}).
  The old one is $O(\sqrt{\log d/(nm)})$ while the correct one in Theorem \ref{theorem: new lasso} is $O(\sqrt{1/(nm)})$.
   The $\log d$ difference can be significant in high dimensional settings.

\item  Theorem \ref{theorem: klopp lasso} involves an additional nuisance term $O(\sqrt{\log d/n})$, similarly to Theorem \ref{theorem: heavy tail}. In the high-dimensional regime where $n\leq m_1m_2$, it is straightforward to verify that the first term–the desirable term–is the larger one. However, in less challenging scenarios in which $n$ is sufficiently large, the nuisance term dominates. This limitation arises from a Frobenius norm based peeling argument that \cite{klopp2014noisy} used to prove restricted strong convexity. In contrast, we have adopted a different peeling scheme introduced in \cite{wainwright2019high}, which peels certain matrix space by the infinity norm and the nuclear norm. As a result, our new bound does not include the nuisance term $O(\sqrt{\log d/n})$.

\item Our new bound requires a slightly stronger condition on $n$. The $\log n$ term arises from the truncation of the sub-Gaussian noise, in order to apply the matrix concentration inequalities from \cite{brailovskaya2024universality}. In the most interesting regime where $n\leq m_1m_2$, the sample size requirement in Theorem \ref{theorem: new lasso} becomes $n\geq Cm\log^5 d$ which is in the typical form of $n\geq m\text{Poly}(\log d)$ in the matrix completion literature. Under both with-replacement and without-replacement settings, the order of $\text{Poly}(\log d)$ often vary with problem assumptions and methods used; see, for example, \cite{koltchinskii2011nuclear,klopp2014noisy,chen2020noisy}. In our results, the leverage of the concentration inequalities \citep{brailovskaya2024universality} causes a small change of the order of $\text{Poly}(\log d)$. This fluctuation is acceptable in light of the substantial gain in sharpness of the resulting bounds. We emphasize that previous concentration inequalities cannot yield bounds as sharp as ours, even when the sample size is fixed at the same level as in our results.

\item 

Compared with Theorem \ref{theorem: klopp lasso}, Theorem \ref{theorem: new lasso} additionally requires Condition~(4) 
in Assumption \ref{assumption:X_i}. Condition~(4) rules out extreme cases where the sampling 
probabilities for certain rows or columns concentrate on only a few entries. 
As discussed following Assumption \ref{assumption:X_i}, this condition is mild. It is needed to 
control the spectral norms of key random matrices via the concentration 
inequalities of \cite{brailovskaya2024universality}, which allows us to obtain 
$\log d$-free bounds.

\end{enumerate}

\subsection{Matrix completion with unknown variance of the noise}\label{subsec:unknown}
\label{mc:unk}
The choice of $\lambda$ in Theorem \ref{theorem: klopp lasso} is dependent on the variance of the noise. In practice, $\sigma$ is often unknown. To address this issue, \cite{klopp2014noisy} proposes the following square-root lasso type estimator and studies its convergence rate.

\begin{equation}\label{estimator: sq lasso}
  \hat{A}_S=\argmin_{\|A\|_\infty
  \leq  a} \left\{ \sqrt{\dfrac{1}{n}\sum_{i=1}^n (Y_i-\langle X_i, A\rangle)^2} +\lambda \|A\|_*  \right\} .
\end{equation}
We obtain a new result for (\ref{estimator: sq lasso}).
\begin{theorem} \label{theorem: new sq lasso}
 Let Assumptions \ref{assumption:X_i} and  \ref{assumption:noise-subgaussian} hold. Suppose $n\geq C_1m \log^4 d~(\log d+\log n)$. 
There exist $\{C_i\}_{i=2}^4$ such that  
 with 
 \begin{equation*}
   \lambda=C_2\sqrt{\frac{1}{nm}},  
 \end{equation*}
  the following bound holds with probability at least $1-\frac{3}{d}$,
\begin{equation}\label{inequality: new sq lasso}
 \dfrac{\|\hat{A}_S-A_0\|_F^2}{m_1m_2}\leq \dfrac{C_3\mu^2\max(a^2,\sigma^2)rM}{n},  
\end{equation} 
provided that   $n\geq C_4\mu Mr $. Here, $C_1$ is a universal constant,  and $\{C_i\}_{i=2}^4$ depend on $\{L_i\}_{i=1}^3$,  where $L_i$'s are constants from Assumptions \ref{assumption:X_i} and \ref{assumption:noise-subgaussian}.
\end{theorem}
We can compare the above bound  with the one obtained in \cite{klopp2014noisy}.
\begin{theorem}[Theorem 10 in \cite{klopp2014noisy}]\label{theorem:squarerootlasso}
 Let Assumption \ref{assumption:noise-subgaussian} and   Conditions (1)-(3) in Assumption \ref{assumption:X_i}  hold.  There exist $\{C_i\}_{i=1}^3$ such that with
 \begin{equation*}
     \lambda= C_1\sqrt{\frac{\log  d}{nm}},
 \end{equation*}
the following upper bound holds with  probability $1-\frac{3}{d}-2\exp(-C_2n)$,
\begin{equation*}
  \dfrac{\|\hat{A}_S-A_0\|_F^2}{m_1m_2}\leq C_3 \max\left\{\dfrac{\mu^2\max(a^2,\sigma^2)  rM\log d}{n},a^2\mu\sqrt{\frac{\log d}{n}} \right\},  
\end{equation*}
provided that $n\geq 8C_1\mu rM\log d$. Here $C_1, C_3$ depend on $L_1, L_2$ and $C_2$ depends on $L_1$, where $L_1, L_2$ are constants from Assumptions \ref{assumption:X_i} and \ref{assumption:noise-subgaussian}.
\end{theorem}
Comparing the convergence rate  in Theorem  \ref{theorem:squarerootlasso} and the lower bound in  (\ref{inequality: minimax lower}), they are not matched due to the existence of  the $\log d$ factor. 
Our new bound in \eqref{inequality: new sq lasso} resolves the issue and shows that the estimator \eqref{estimator: sq lasso} is indeed minimax rate optimal. Regarding more comparisons between Theorems \ref{theorem: new sq lasso} and \ref{theorem:squarerootlasso}, we refer to remarks after Theorem \ref{theorem: klopp lasso} for similar comments. 

\section{Proofs}\label{section: proofs}

The matrix completion problem has been extensively studied in the literature, and its analytical framework is now well established. Before delving into the details of the proofs, we provide a discussion clarifying the relationship between our proofs and those in the existing literature, as well as highlighting the new ingredients of the current work.

The three estimators discussed in our paper share a similar structure.
To illustrate the main ideas,  we can write them in a unified form:
\begin{equation*}
    \Ah_G=\argmin_{\|A\|_\infty\leq a} \left\{ \Phi(A)+\lambda\|A\|_*\right\},
\end{equation*}
where $\Phi(\cdot)$ is a convex loss function.
% let us take estimator (\ref{estimator:lasso})  as an example
% \begin{equation*}
% \hat{A}=\argmin_{\|A\|_\infty\leq a}  \left\{ \dfrac{1}{n}\sum_{i=1}^n (Y_i-\langle X_i,A\rangle)^2+\lambda \|A\|_*  \right\}.
% \end{equation*}
A typical analysis proceeds in two steps:
\begin{itemize}
    \item[(1)] \emph{Derivation of  nearly low-rank structure}. One needs to show
\begin{equation*}
    \|\hat{A}_G-A_0\|_*\leq C_1\sqrt{r}\|\hat{A}_G-A_0\|_F,
\end{equation*}
which  holds under the condition
\begin{equation*}
    \lambda\geq C_2\|\nabla\Phi(A_0)\|,
\end{equation*}
where $\|\cdot\|$ denotes the spectral norm. $\|\nabla\Phi(A_0)\|$ is often related to the following class of random matrices
\begin{equation*}
    \left\|\sumi z_iX_i\right\|,
\end{equation*}
where $z_i$  are bounded or sub-Gaussian random variables. 
An appropriate choice of $\lambda$ requires  sharp non-asymptotic bounds on $\|\sumi z_iX_i\|$.

\vspace{0.1cm}

\item[(2)] \emph{Establishment of  restricted strong convexity}. One needs to show 
\begin{equation*}
    \la \nabla\Phi(\Ah_G)-\nabla\Phi(A_0), \Ah_G-A_0\ra\geq C\|\Ah_G-A\|^2_{w(F)}-\Upsilon,
\end{equation*}
where $\Upsilon$ is some error term and $\|A\|_{w(F)}=\sqrt{\sumij a^2_{ij}P_{ij}}$ for $A=(a_{ij})\in \mathbb{R}^{m_1\times m_2}$. This step requires bounding some empirical processes. And one will need to bound the following quantity (or its variants), 
\begin{equation*}
    \Eb\left[\left\|\sumi\epsilon_iX_i\right\|\right],
\end{equation*}
which is related to some Rademacher complexity. Here, $\{\epsilon_i\}_{i=1}^n$ are symmetric Bernoulli random variables independent of $\{X_i\}_{i=1}^n$.
% One needs to show
% \begin{equation*}
%     \sup_{\hat{A}-A_0\in\Omega} \left|\sumi\la X_i,\hat{A}-A_0\ra^2-\|\hat{A}-A_0\|^2_{w(F)}\right|,
% \end{equation*}
% where $\|A\|_{w(F)}=\sqrt{\sumij a^2_{ij}P_{ij}}$  and $\Omega$ is certain matrix space.
% This step requires  bounding
% \begin{equation*}
%     \Eb\left[\left\|\sumi\epsilon_iX_i\right\|\right],
% \end{equation*}
% where $\{\epsilon_i\}_{i=1}^n$ are symmetric Bernoulli random variables independent of $\{X_i\}_{i=1}^n$. 
\end{itemize}
Existing works adopt standard concentration inequalities (e.g., Theorem 6.17 in \cite{wainwright2019high}), obtaining
 \begin{equation}\label{b1}
  \left\|\sumi z_iX_i\right\|\lesssim \sqrt{\frac{\log d}{nm}},\quad 
  \Eb\left[\left\|\sumi\epsilon_iX_i\right\|\right]\lesssim \sqrt{\frac{\log d}{nm}}.
\end{equation}
The  $\log d$ factor in the sub-optimal convergence rate essentially results from the loose bounds (\ref{b1}).

To overcome this limitation, we leverage new advances in matrix concentration inequalities 
\cite{brailovskaya2024universality}, to obtain sharper bounds 
\begin{equation*}
    \left\|\sumi z_iX_i\right\|\lesssim \sqrt{\frac{1}{nm}}, \quad 
 \Eb\left[\left\|\sumi\epsilon_iX_i\right\|\right]\lesssim \sqrt{\frac{1}{nm}}.
\end{equation*}
The concentration inequalities in \cite{brailovskaya2024universality} are not directly applicable to our problems due to certain uniform boundedness assumptions. We employ a truncation scheme together with careful non-asymptotic calculations, to ensure a successful adaptation of these advanced concentration inequalities. The sharp bounds for the spectral norms are presented in Section \ref{section:spn}.

Another technical issue arises from bounding some empirical process deviations in Step (2). Commonly adopted peeling arguments (e.g., those in \cite{yu2024low}) often introduce an additional nuisance error term of order $O(\sqrt{(\log d)/n})$, which could be dominant if the sample size exceeds a threshold. In this paper,  inspired by \cite{wainwright2019high},  we develop a new peeling argument for the Huber loss function to reduce the error from $O(\sqrt{(\log d)/n})$ to $O((\log d)/n)$, so that the corresponding nuisance term becomes negligible without an upper bound constraint on the sample complexity. The same improvement applies to the quadratic loss, which follows as a corollary of the corresponding results for the Huber loss. Related results can be found in  Lemmas \ref{lemma: RSC Huber} and \ref{lemma:rsc L2}.

The sharp spectral norm analysis, together with other technical refinements, removes the $\log d$ factor in the convergence rate and establishes the minimax optimality.

In this section, $\{c_i\}_{i=0}^4$  are specific constants whose values do not change at each occurrence.

\subsection{Proof of Theorem \ref{theorem: new heavy tail}} \label{P2.1} %proof of heavy tail matrix completion

In this section, the derivative of $l_\tau(\cdot)$ is denoted by $\phi_\tau(\cdot)$, where
\begin{equation*}
 \phi_\tau(x)=\begin{cases}
   x, |x|\leq \tau,\\
   \sign(x)\tau, |x|>\tau.
 \end{cases}   
\end{equation*}
We set
\begin{equation*}
    L(A)=\sumi l_\tau(Y_i-\la X_i,A\ra).
\end{equation*}

We now provide a general form of the convergence rate.
\begin{lemma}\label{lemma: general heavy tail}
Let Assumptions  \ref{assumption:X_i} and \ref{assumption: mean var} hold.
Suppose $n\geq 160m\log ^4 d$ and $\tau=\frac{\max(\sigma,a)}{\log^2 d}\sqrt{\frac{n}{m}}$. 
If 
\begin{equation*}
\lambda\geq 3\left\|\sumi \phi_{\tau}(\xi_i) X_i\right\|,    
\end{equation*}
 then the following bound holds with probability at least $1-\frac{1}{d}$, 
 \begin{equation*}
   \frac{\V\dif\V_F^2}{m_1m_2}\lesssim\mu^2rm_1m_2(\lambda^2+a^2  (\Eb[\|\Rc\|])^2  )+\frac{\mu^2 a^2\log  d }{n}. 
 \end{equation*}
Here $\Rc$ is defined as
  $ \Rc=\sumi \epsilon_i X_i \ind_{\chi_i}, $ 
 where $\chi_i=\{|\xi_i|\leq \tau/2 \}$ and $\epsilon_i$ are symmetric Bernoulli random variables independent of $\{(X_i,Y_i)\}_{i=1}^n$.
\end{lemma}%end of proof of lemma: general heavy tail

\begin{proof}[\textbf{Proof of Theorem \ref{theorem: new heavy tail}}]

\textbf{Part I ($\xi_i|X_i$ are symmetric)}:
Recall that  $ \tau=\frac{\max(\sigma, a)}{\log^2d}\sqrt{\frac{n}{m}} $. 
By (\ref{inequality:nasym-spectral-bounded}) in Lemma \ref{lemma: sharp spectral},
the following holds   with probability at least $1-\frac{1}{d}$, 
\begin{equation*}
\left\|\sumi \phi_\tau(\xi_i) X_i\right\|\leq C_1\max(\sigma,a)\sqrt{\dfrac{1}{nm}},
\end{equation*}
where $C_1$ depends on $L_2$ and $L_3$.
With this specific $C_1$, we set
\begin{equation*}
    \lambda=3C_1\max(\sigma,a)\sqrt{\frac{1}{nm}},
\end{equation*}
to ensure $\lambda\geq 3\left\|\sumi \phi_{\tau}(\xi_i) X_i\right\|$.
By (\ref{inequality: E-spectral heavy}) in Lemma \ref{lemma: sharp spectral}, we have
\begin{align}\notag\label{f2}
\Eb\left[\left\Vert \dfrac{1}{n}\sumn \epsilon_iX_i\ind_{\chi_{i}}\right\Vert\right]\leq C_2\sqrt{\frac{1}{nm}},
\end{align}
where $C_2$ depends on $L_2$ and $L_3$.
Substituting this bound and the above choice of $\lambda$
into the inequality of Lemma \ref{lemma: general heavy tail}, we complete the proof of  Part I in Theorem \ref{theorem: new heavy tail}.

\textbf{Part II (general distributional case)}: 

Now we  consider the case that $\xi_i$ only has finite second moments but its distribution is not necessarily symmetric. 
Without  the condition that $\xi_i$ are symmetric,  $\nabla L(A_0)$ may no longer be mean-zero.
It becomes necessary to control the spectral norm of  $\Eb[\nabla L(A_0)]$.
We  write:
\begin{equation*}
    \spx=\sumi (\phi_{\tau}(\xi_i)-\Eb[\phi_{\tau}(\xi_i)|X_i])X_i+\sumi \Eb[\phi_{\tau}(\xi_i)|X_i]X_i.
\end{equation*}
Denote $\Eb[\phi_{\tau}(X_i)|X_i]$ by $u_i$ (we omit the dependence on $\tau$ to ease the notation).
 (\ref{inequality:nasym-spectral-bounded}) in Lemma \ref{lemma: sharp spectral} still applies to $\sumi (\phi_{\tau}(\xi_i)-u_i)X_i$ given $\tau=\max(\sigma,a)\sqrt{\frac{n}{m\log^4 d}}$.
  Then, we have
 \begin{equation*}
    \left\| \sumi (\phi_{\tau}(\xi_i)-\Eb[\phi_{\tau}(\xi_i)|X_i])X_i\right\|\leq C_1\max(\sigma, a)\sqrt{\frac{1}{nm}}.
 \end{equation*}
 Next, we need to bound
\begin{align*}
\left\| \frac{1}{n}\sum_{i=1}^nu_iX_i   \right\|&=\left\| \dfrac{1}{n}\sum_{i=1}^n (u_iX_i-\Eb[u_iX_i])+\sumi \Eb[u_iX_i]\right\|\\
&\leq \left\| \dfrac{1}{n}\sum_{i=1}^n (u_iX_i-\Eb[u_iX_i])\right\|+\lnorm \sumi\Eb[u_iX_i]\rnorm.
\end{align*}

Consider  a random variable $Y$ has 0 mean  and finite moment of order $2+\kappa$ with $(\Eb[|Y|^{2+\kappa}])^{\frac{1}{2+\kappa}}\leq \sigma$ where $\kappa\geq 0$. 
By  H\"older's inequality,  for  $\tau>0$, we have
\begin{align*}
    |\Eb \phi_{\tau}(Y)|&=|\Eb [Y-\phi_{\tau}(Y)]|\\
    &=|\Eb[(Y-\phi_{\tau}(Y)) \ind(|Y|>\tau)]\\
    &\leq (\Eb[|Y|^p])^{1/p}\Pb^{1/q}(|Y|>\tau)
\end{align*}
 where $1/p+1/q=1$ and $p,q\geq 1$.
By Markov's inequality, we have $\Pb(Y>\tau)\leq \Eb[|Y|^{2+\kappa}]/\tau^{2+\kappa}$. 
Take $1/q=1-1/(2+\kappa)$.
Then, we obtain  
\begin{align}
\notag  |\Eb[(Y-\phi_{\tau}(Y))\ind(|Y|>\tau)]|&\leq (\Eb[|Y|^{2+\kappa}])^{\frac{1}{2+\kappa}}\left((\Eb[|Y|^{2+\kappa}])^{\frac{1+\kappa}{2+\kappa}}\tau^{-(1+\kappa)}\right)\\
\label{v12-4}  &\leq \sigma\left(\frac{\sigma}{\tau}\right)^{1+\kappa}.
\end{align}

Assume that $(\Eb[|\xi_i|^{2+\kappa}])^{\frac{1}{2+\kappa}} \leq  \sigma$ 
uniformly  given $X_i$.
Applying  (\ref{v12-4})  with $\tau=\max(\sigma,a) \sqrt{\frac{n}{m\log^4 d}} $ to $u_i$, we obtain
\begin{equation}\label{v12-5}
  |u_i|\leq \sigma \left(\frac{m\log^4 d}{n}\right)^{(1+\kappa)/2}. \end{equation}

We next apply  Theorem 6.1 in \cite{tropp2012user} to $\frac{1}{n}\sum_{i=1}^n (u_iX_i-\Eb[u_iX_i]) $.
For readers' convenience, we reproduce it below.
\begin{theorem}[Theorem 6.1 in \cite{tropp2012user}: Matrix Bernstein - bounded case]\label{theorem:trad}
Consider a finite sequence $\{Q_i\}_{i=1}^n$ of independent, random matrices with dimension $m_1\times m_2$. Assume that
\begin{equation*}
    \Eb [Q_k] = 0 \quad \text{and} \quad \|Q_k\| \le R \quad \text{almost surely}.
\end{equation*}
Define 
\begin{equation*}
    \gamma^2(Q) =\max (\|\Eb[QQ^T]\|, \|\Eb[Q^TQ]\|).
\end{equation*}
Then the following chain of inequalities holds for all $t \ge 0$:
\begin{equation*}
  \mathbb{P} \left( \|Q\| \ge t \right)
\le d \cdot \exp\left( \frac{-t^{2}/2}{\sigma^{2} + R t/3} \right).  
\end{equation*}   
\end{theorem}
By Theorem \ref{theorem:trad},  there exist $C_1$ and $C_2$ such that 
\begin{equation}\label{inequality: trad}
\Pb(\|Q\|\geq C_1R(t+\log d)+C_2\sigma\sqrt{t+\log d})\leq \exp(-t).  
\end{equation}
Consider $Q_i=\frac{1}{n}(u_iX_i-\Eb[u_iX_i])$.

By (\ref{v12-5}), we have
\begin{equation*}
 \frac{1}{n}\|u_iX_i-\Eb[u_iX_i]\|\leq   \frac{2\sigma}{n} \left(\frac{m\log^4 d}{n}\right)^{(1+\kappa)/2}. 
\end{equation*}
And we also have
\begin{equation*}
 \|\Eb(u_iX_i-\Eb[u_iX_i])(u_iX_i-\Eb[u_iX_i])^T\|\lesssim \frac{L_2\sigma^2}{m}\left(\frac{m\log^4 d}{n}\right)^{(1+\kappa)}   
\end{equation*}
and 
\begin{equation*}
\|\Eb(u_iX_i-\Eb[u_iX_i])^T(u_iX_i-\Eb[u_iX_i])\|\lesssim \frac{L_2\sigma^2}{m}\left(\frac{m\log^4 d}{n}\right)^{(1+\kappa)}.    
\end{equation*}
Thus,  by (\ref{inequality: trad}), 
the following holds  with probability at least $1-\frac{1}{d}$,  
\begin{align*}
 \left\|\frac{1}{n}\sum_{i=1}^n (u_iX_i-\Eb[u_iX_i])\right\|&\lesssim \sigma \sqrt{\frac{\log d}{nm}}\left(\frac{m\log^4 d}{n}\right)^{(1+\kappa)/2}\\
 &+ \frac{\sigma\log d}{n}\left(\frac{m\log^4 d}{n}\right)^{(1+\kappa)/2}.   
\end{align*}
Given $n\geq m\log^{4+\frac{1}{\kappa+1}} d$, we have, with probability at least $1-\frac{1}{d}$,
\begin{equation*}
\left\| \dfrac{1}{n}\sum_{i=1}^n (u_iX_i-\Eb[u_iX_i])\right\|\lesssim  \sigma\sqrt{\frac{1}{nm}}.   
\end{equation*}
Denote $u_i(j,k)=\Eb[\phi_\tau(\xi_i)|X_i=e_je^T_k]$.
Note
\begin{align*}
\|\Eb[u_iX_i]\|&=\sup_{\|z\|_2=1,\|y\|_2=1}z^T\Eb[u_iX_i]y\\
&\leq \left( \sum_{j=1}^{m_1}\sum_{k=1}^{m_2}P^2_{jk}\right)^{1/2}\left(\sum_{j=1}^{m_1}\sum_{k=1}^{m_2} z_j^2y_k^2u^2_i(j,k)\right)^{1/2}\\
&\leq \sigma\sqrt\frac{1}{m_1m_2}  \left(\frac{m\log^4 d}{n}\right)^{(1+\kappa)/2}.
\end{align*}

Given $Mn^{\kappa}\geq m^{1+\kappa}\log^{4(1+\kappa)} d$,  
we have
\begin{equation*}
   \left\|\sumi\Eb[u_iX_i]\right\|\leq  \sigma\sqrt{\frac{1}{nm}}.
\end{equation*}
Combining the pieces, we have the following results:
(i) when $\kappa=0$, given $M\geq  m \log^4 d$ and $n\geq 160 m\log^5 d$, the following holds with probability at least $1-\frac{2}{d}$,
\begin{equation*}
    \left\|\sumi \phi_\tau(\xi_i) X_i\right\|\lesssim\max(\sigma,a)\sqrt{\frac{1}{nm}}.
\end{equation*}
(ii) when $k>0$, given  $n\geq m \log^{4+\frac{4}{\kappa}} d$ and $n\geq 160 m\log^{4+\frac{1}{1+\kappa}} d$, the following holds with probability at least $1-\frac{2}{d}$,
\begin{equation*}
 \left\|\sumi \phi_\tau(\xi_i) X_i\right\|\lesssim\max(\sigma,a)\sqrt{\frac{1}{nm}}.   
\end{equation*}
\end{proof}%end of proof of Theorem heavy tail

  Before presenting the proof of Lemma \ref{lemma: general heavy tail}, we state a technical lemma. This lemma shows that with a fine-tuned $\lambda$ in (\ref{estimator: heavy tail}), $\Ahh-A_0$ lies in a nearly low-rank region.

\begin{lemma}\label{lemma: special cone}
For a matrix $A$, denote its singular value decomposition as $A=U\Sigma V^T$. We define the projection operators
\begin{equation*}
 P_{A}(B)=UU^T B V V^T   
\end{equation*}
and
\begin{equation*}
 P_{A}^{\perp}(B)=B-P_{A}(B).   
\end{equation*}
With $\lambda$ in (\ref{estimator: heavy tail}) satisfying $\lambda \geq 3 \|\nabla L(A_0) \|$, we  obtain
\begin{equation*}
   \V P_{A_0}^{\perp}(\Ahh-A_0)\|_*
   \leq 2\V P_{A_0}(\Ahh-A_0)\|_*. 
\end{equation*}
\end{lemma} %end of lemma of lowrank property

\begin{proof}%Proof of nealry low rank property
As the solution to a convex program, by Proposition 1.3 in \cite{bubeck2014theory}, $ \Ahh $ satisfies
\begin{align*}
\la \nabla L(\Ahh)+\lambda Z, \Ahh-A_0\ra\leq 0,
\end{align*}
where $ Z\in \partial \Vert \Ahh\Vert_* $.
Since $ \Vert \cdot\Vert_* $ is a convex function, we have
\begin{equation*}
\la Z, A_0-\Ahh \ra \leq \Vert A_0\Vert_*-\Vert\Ahh\Vert_*. 
\end{equation*}
By the convexity of $ L(\cdot) $,
\begin{equation*}
\la \nabla L(\Ahh)-\nabla L(A_0), \Ahh-A_0\ra\geq 0.
\end{equation*}
Combining the results above, we obtain
\begin{align}
\notag0&\leq \la \nabla L(\Ahh)-\nabla L(A_0), \Ahh-A_0\ra\\
&\leq \lambda (\Vert A_0\Vert_*-\Vert\Ahh\Vert_*)
+\Vert \nabla L(A_0)\Vert \Vert\Ahh-A_0\Vert_*.\label{inequality: lemma nl 1}
\end{align}
Observe that
\begin{align*}
\V \Ahh\|_*&=\V P_{A_0}(\Ahh-A_0)+P_{A_0}^{\perp}(\Ahh-A_0)+P_{A_0}(A_0)\|_*\\
&\geq \V P_{A_0}^{\perp}(\Ahh-A_0)\|_*+\V P_{A_0}(A_0)\|_*-\V P_{A_0}(\Ahh-A_0)\|_*.
\end{align*}
Then, it follows that
\begin{align*}
&0\leq \frac{4}{3}\V P_{A_0}(\Ahh-A_0)\|_*-\frac{2}{3}\V P_{A_0}^{\perp}(\Ahh-A_0)\|_*,
\end{align*} 
which leads to the statement of Lemma \ref{lemma: special cone}.
\end{proof}
%end of proof of nearly lowrank property

\begin{proof}[\textbf{Proof of Lemma \ref{lemma: general heavy tail}}]

By Lemma \ref{lemma: special cone}, $\Ahh-A_0$ satisfies 
\begin{equation*}
    \|\Ahh-A_0\|_*
\leq\\ 3\|P_{A_0}(\Ahh-A_0)\|_*.
\end{equation*}
From \cite{wainwright2019high} Chapter 10, we know, for a matrix $B$,
\begin{equation*}
   \| P_{A_0}(B)\|_*\leq\sqrt{2\text{rank}(A_0)}\|B\|_F.
\end{equation*}
Then we have
\begin{equation*}
    \|\hat{A}_H-A_0\|_*\leq 3\sqrt{2r}\|\hat{A}_H-A_0\|_F.
\end{equation*}
Define the weighted Frobenius norm as
\begin{equation*}
    \|A\|_{w(F)}=\sqrt{\sum_{i=1}^{m_1}\sum_{j=1}^{m_2}P_{ij}a^2_{ij}}.
\end{equation*} 

As a solution of the  convex program (\ref{estimator: heavy tail}), by (\ref{inequality: lemma nl 1}),  $\Ah_H $ satisfies:
\begin{align}
\notag\la \nabla L(\Ah_H)-L(A_0), \Ah_H-A_0\ra &\leq \frac{4}{3}\lambda\|\Ah_H-A_0\|_*\\
& \leq 4\sqrt{2r}\lambda\|\Ahh-A_0\|_{F},\label{AHeq1}
\end{align}
given $\lambda\geq 3\left\|\sumi \phi_{\tau}(\xi_i) X_i\right\|$.

Now we state a technical lemma.
\begin{lemma}\label{lemma: RSC Huber}
 With probability at least $ 1-\frac{1}{d} $, the following holds uniformly for matrices with $\|A\|_{w(F)}\leq s$
\begin{align}
\notag  \la\nabla L(A_0+A)-\nabla L(A_0), A\ra
&\geq
\V A\Vf^2\left(\frac{9}{10}
-\dfrac{4\sigma^2}{\tau^2}\right)-\dfrac{16\|A\|^2_\infty s^2}{\tau^2}\\
\label{3.21} &-c_3\Bigg(\|A\|_\infty \V A\|_*\Eb[\|\Rc\|]+\frac{\mu \|A\|_\infty^2\log d}{n}\Bigg),
\end{align}
 where $ c_3 $ is a universal constant. 
\end{lemma}%end of lemma RSC property for Huber loss

 We assume $\|\Ahh-A_0\|_{w(F)}\leq s$ where s will be determined latter. 
By Lemma \ref{lemma: RSC Huber},  the following holds with probability at least $1-\frac{1}{d}$,
\begin{align}
\notag  &\la\nabla L(\Ahh )-\nabla L(A_0), \Ahh-A_0\ra\\
\geq
\notag&\V \Ahh-A_0\Vf^2\left(\frac{9}{10}
-\dfrac{4\sigma^2}{\tau^2}\right)-\dfrac{16\|\Ahh-A_0\|^2_\infty s^2}{\tau^2}\\
\label{v12-3} &-c_3\Bigg(\|\Ahh-A_0\|_\infty \V \Ahh-A_0\|_*\Eb[\|\Rc\|]+\frac{\mu \|\Ahh-A_0\|_\infty^2\log d}{n}\Bigg),
\end{align}
where $ c_3 $ is a universal constant. 
By (\ref{AHeq1}) and (\ref{v12-3}), we have
\begin{align}
\notag  \V \Ahh-A_0\Vf^2\left(\frac{9}{10}
-\dfrac{4\sigma^2}{\tau^2}\right)\leq  4\sqrt{2}\lambda\|\Ahh-A_0\|_{F}+ \dfrac{16\|\Ahh-A_0\|^2_\infty s^2}{\tau^2}\\
+c_3\Bigg(\|\Ahh-A_0\|_\infty \V \Ahh-A_0\|_*\Eb[\|\Rc\|]+\frac{\mu \|\Ahh-A_0\|_\infty^2\log d}{n}\Bigg). \label{Aheq3}
\end{align}
Given $n\geq 160m\log ^4 d$, we have
\begin{align*}
    \frac{4\sigma^2}{\tau^2}&\leq \frac{1}{10},\\
    \frac{64a^2}{\tau^2}&\leq \frac{2}{5}.
\end{align*}
Since $\|\Ahh-A_0\|_\infty\leq 2a$ and $\|\hat{A}_H-A_0\|_*\leq 3\sqrt{2r}\|\hat{A}_H-A_0\|_F$, we have
\begin{align*}
    \frac{\mu \|\Ahh-A_0\|_\infty^2\log d}{n}&\leq \frac{4\mu a^2\log d }{n},\\
    \|\Ahh-A_0\|_\infty \V \Ahh-A_0\|_*\Eb[\|\Rc\|]&\leq 180c_3 \mu a^2 rm_1m_2(\Eb[\|\Rc\|])^2 +\frac{\|\Ahh-A_0\|_F^2}{10c_3\mu m_1m_2},\\
    \dfrac{16\|\Ahh-A_0\|^2_\infty s^2}{\tau^2}&\leq\frac{64a^2s^2}{\tau^2}.
\end{align*}
Moreover, it holds
\begin{align*}
 4\sqrt{2r}\lambda\|\Ahh-A_0\|_F
 \leq 160 \mu rm_1m_2\lambda^2 +\frac{1}{5}\|\Ahh-A_0\|^2_{w(F)}.
\end{align*}
Plug in the above results in (\ref{Aheq3}), then we obtain that there exists a constant $c_4$ such that the following holds
\begin{align}\label{AHeq2}
\|\dif\|^2_{w(F)} 
\leq \frac{2}{3} s^2+c_4(\mu rm_1m_2(\lambda^2+a^2(\Eb[\|\Rc\|])^2 )+\frac{\mu a^2\log  d }{n}).
\end{align}

Let $\delta=c_4\left(\mu rm_1m_2(\lambda^2+a^2 (\Eb[\|\Rc\|])^2 )+\frac{\mu a^2\log  d }{n}\right)$.
We take $ s^2= 10\delta $.
If  $\|\hat A_H - A_0\|_{w(F)} > s$, we define $\tilde A_H$ by
\[
\tilde A_H - A_0 = \eta(\hat A_H - A_0),
\]
where $\eta > 0$ is chosen as $s/\|\Ahh-A_0\|_{w(F)}$  so that
\[
\|\tilde A_H - A_0\|_{w(F)} = s.
\]

To obtain (\ref{AHeq2}), we start from (\ref{AHeq1}) and (\ref{v12-3}) and use the following properties of $\Ahh-A_0$
\begin{itemize}
    \item[(a1)] $\|\Ahh-A_0\|_\infty\leq 2a$; 
    \item[(a2)] $\|\Ahh-A_0\|_*\leq 3\sqrt{2r}\|\Ahh-A_0\|_F $ ;
    \item[(a3)] $\|\Ahh-A_0\|_{w(F)}\leq s$.
\end{itemize}
By  Lemma F.2 in  \cite{fan2018lamm}, we have
\begin{align*}
	\la \nabla L(\At_H)-\nabla L(A_0),\At_H-A_0\ra &\leq \eta\la \nabla L(\Ahh)-\nabla L(A_0),\dif\ra \\
&\leq\frac{4}{3}\eta\lambda\V ( \Ahh -A_0)\|_*\\
&=\frac{4}{3}\lambda\|\tilde A_H-A_0\|_*.
\end{align*}
Thus, (\ref{AHeq1}) holds for $\tilde{A}_H$. 
By Lemma \ref{lemma: RSC Huber}, (\ref{v12-3}) holds for $\tilde{A}_H$, too.
Note that $ \At_H $  satisfies 
\begin{equation*}
\|\tilde{A}_H-A_0\|_*\leq 3\sqrt{2r}\|\At_H-A_0\|_F.  
\end{equation*}
Hence $\tilde{A}_H-A_0$ also  satisfies (a1)-(a3).
Now we can repeat the argument for $\tilde{A}-A_0$ and obtain 
\begin{equation*}
\V \At_H-A_0\Vf^2\leq\delta+\frac{2}{3} s^2\leq 0.77s^2,
\end{equation*} 
which contradicts with the fact that $ \V\At_H-A_0\Vf^2=s^2 $.

Therefore,  we obtain that the following holds with probability at least $1-\frac{1}{d}$,
\begin{equation*}
\V\dif\Vf^2\leq 10c_4\left(\mu rm_1m_2(\lambda^2+a^2(\Eb[\|\Rc\|])^2)+\frac{\mu a^2\log  d }{n}\right)
\end{equation*}
and 
\begin{equation*}
\frac{\V\dif\V_F^2}{m_1m_2}\lesssim\mu^2 rm_1m_2(\lambda^2+a^2(\Eb[\|\Rc\|])^2)+\frac{\mu^2 a^2\log  d }{n}.
\end{equation*}
\end{proof}%end of proof for lemma: genral heavy tail 

\subsection{Proof of restricted strong convexity for Huber loss function} \label{section:rsc huber}
%proof of restricted strong convexity for Hube loss function
\begin{proof}[\textbf{Proof of Lemma \ref{lemma: RSC Huber}}]
Recall $L(A)=\sumi l_\tau(Y_i-\la X_i,A\ra).$
Consider events $ H_i=\{ |\la X_i, A \ra |\leq \tau/2\}\cap\{ |\xi_i|\leq \tau/2\} $. 
We have the following result
\begin{align*}
&\quad\la \nabla L(A_0+A)-\nabla L(A_0), A \ra\\
&=\sumi (\phi_\tau(Y_i-\la X_i,A_0\ra)-\phi_\tau(Y_i-\la X_i, A_0+A\ra) ) \la X_i, A\ra \\
&\geq \sumi (\phi_\tau(Y_i-\la X_i,A_0\ra)-\phi_\tau(Y_i-\la X_i,A_0+A\ra) ) \la X_i, A\ra\ind_{H_i}\\
&= \dfrac{1}{n}\sumn \la X_i, A\ra^2\ind_{H_i}.
\end{align*}
Define  $ G_{R}(x):\Rb\rightarrow\Rb$ by
\begin{align*}
G_R(x)=\begin{cases}
x^2, |x|\leq R,\\
(x-2R\cdot \sign(x))^2, R\leq |x|\leq 2R,\\
0, \text{otherwise}.
\end{cases}
\end{align*}
$ G_R(x) $ has the following properties:
\begin{itemize}
\item[1.] For $x\in \Rb$, 
\begin{equation}\label{Prop1}
  x^2\ind_{|x|\leq R}\leq  G_R(x)  \leq  x^2\ind_{|x|\leq 2R}.
\end{equation}
\item[2.] $ G_R(x) $ is  Lipschitz continuous  with $ |G_R(x)-G_R(y)|\leq 2R|x-y| $ since $ |G_R'(x)|\leq 2R $. 
Moreover, if for  $ |x|,|y|\leq  a $,
\begin{equation}\label{Prop2}
    |G_R(x)-G_R(y)|\leq  2a|x-y|.
\end{equation}
\end{itemize}
Then by (\ref{Prop1}) and  $ \V A\Vf\leq s $, we have
\begin{align*}
&\quad~ \left\la X_i, \dfrac{A}{\V A\Vf}\right\ra^2\ind_{H_i}\\
&=\left\la X_i, \dfrac{A}{\V A\Vf}\right\ra^2
\ind\left(\dfrac{\la X_i,A\ra}{\V A\Vf}\leq \dfrac{\tau}{2\V A\Vf}\right)
\ind_{\chi_i}\\
&\geq \left\la X_i, \dfrac{A}{\V A\Vf}\right\ra^2 
\ind\left(\dfrac{\la X_i, A\ra}{\V A\Vf}\leq \dfrac{\tau}{2s}\right)
\ind_{\chi_i}\\
& \geq  G_{\frac{\tau}{4s}}\left(\left\langle X_{i},\frac{A}{\|A\|_{w(F)}}\right\rangle\right)\ind_{\chi_{i}},
\end{align*}
where $ \chi_i=\{|\xi_i|\leq \tau/2\} $. Hence, we have
\begin{align*}
\dfrac{1}{n}\sumn \la X_i, A \ra^2\ind_{H_i}
&= \dfrac{\Vert A\Vert_{w(F)}^2}{n}\sumn\left\la X_i, \dfrac{A}{\|A\|_{w(F)}}\right\ra^2\ind_{H_i}\\
&\geq \|A\|^2_{w(F)}\left(\dfrac{1}{n}\sumn G_{\frac{\tau}{4s}}\left(\left\langle X_{i},\frac{A}{\|A\|_{w(F)}}\right\rangle\right)\ind_{\chi_{i}}\right).
\end{align*}
For a matrix $A\in \mathbb{R}^{m_1\times m_2}$, we define
\begin{equation*}
	f(A):=\Eb_{X_i,\xi_i} [G_{\frac{\tau}{4s}} (X_i,A)\ind_{\chi_i}].
\end{equation*}
The following holds for $f(A)$
\begin{align*}
f(A)&\geq \Eb[\la X_i, A\ra^2] \ind\{|\la X_i, A\ra|\leq \tau/4s\}\ind_{\chi_i}\\
&\geq \Eb[\la X_i,A\ra^2]
-\Eb[\la X_i,A\ra^2\ind\{|\la X_i, A\ra|> \tau/4s\}]
-\Eb[\la X_i,A\ra^2 \ind\{|\xi_i|>\tau/2\}]\\
&\geq \V A\Vf^2
-\dfrac{16s^2\Vert A\Vert^2_{\infty}}{\tau^2}\V A\Vf^2
-\dfrac{4\sigma^2}{\tau^2}\V A\Vf^2\\
&=\V A\Vf^2\left(1-\dfrac{16s^2\Vert A\Vert_{\infty}^2}{\tau^2}-\dfrac{4\sigma^2}{\tau^2}\right).
\end{align*}
Thus 
\begin{align}\label{3.22}
&\Vert A\Vf^2f\left(\frac{A}{\Vert A\Vf}\right)\geq \V A\Vf^2(1-\dfrac{4\sigma^2}{\tau^2})-\dfrac{16\|A\|^2_\infty s^2}{\tau^2}.
\end{align}
Now we introduce a useful lemma.
\begin{lemma}\label{lemma: RSC Huber 1}
Assume $ \V A\Vf^2=1 $. Then there are universal constants $ c_0, c_1, c_2 $ such that 
\begin{align}
&\notag \left| \dfrac{1}{n}\sumn G_{\frac{\tau}{4s}}(\la X_i,A\ra )\ind_{\chi_i}-f(A)\right|\\
&\leq 8c_0\|A\|_\infty\|A\|_*\Eb[\|\Rc\|]
+2\sqrt{2}c_1\|A\|_\infty\sqrt{\frac{\mu \log d}{n}}
+8c_2\frac{\mu\V A\V_{\infty}^2\log d}{n} \label{3.23}
\end{align}
for all such $A$, uniformly with probability at least $ 1-\frac{1}{d} $. 
Here  $\Rc=\sumi \epsilon_i X_i \ind_{\chi_i}$
,$\chi_i=\{|\xi_i|\leq \tau/2\}$ and $\{\epsilon_i\}_{i=1}^n$ are independent symmetric Bernoulli random variables independent of $\{(X_i,Y_i)\}_{i=1}^n$.
\end{lemma}% end of lemma 1 RSC Huber 
Combining  (\ref{3.22}) and (\ref{3.23}), we get
\begin{align}
\notag &\quad  \la\nabla L(A_0+A)-\nabla L(A_0), A\ra
\\
\notag&\geq \V A\Vf^2(1-\dfrac{4\sigma^2}{\tau^2})-\dfrac{16\|A\|_\infty s^2}{\tau^2}
-\Bigg( 8c_0\|A\|_\infty\|A\|_* \Eb[\|\Rc\|]\\
\notag&+8c_2\frac{\mu\V A\V_{\infty}^2\log d}{n}+2\sqrt{2}c_1\V A\V_{\infty}\V A\Vf\sqrt{\frac{\mu\log d}{n}}\Bigg)\\
\notag&\geq \V A\Vf^2(1-\dfrac{4\sigma^2}{\tau^2})-\dfrac{16\|A\|^2_\infty s^2}{\tau^2}\\
\label{3.24}&-c_3\Bigg(\|A\|_*\|A\|_\infty
+\frac{\mu \|A\|^2_\infty \log d}{n}\Bigg)
-\frac{\V A\Vf^2}{10}
\end{align}
where in the last step, we apply Young's inequality $ 2ab\leq \frac{a^2}{c^2}+c^2b^2 $ for $a,b,c \in \Rb$ and choose a appropriate constant $ c_3 $. 
Then (\ref{3.21}) is concluded  from (\ref{3.24}).
\end{proof}

\begin{proof}[\textbf{Proof of Lemma \ref{lemma: RSC Huber 1}}]
We begin by introducing a classical concentration inequality for empirical processes, which can be found in many textbooks such as \cite{wainwright2019high}.
\begin{lemma}[Theorem 3.27 in \cite{wainwright2019high}]\label{lemma: CI for EPs}
Let $ Z_i, 1\leq i\leq n $, be independent but not necessarily identically distributed random variables taking values in measurable spaces $ \mathcal{Z}, 1\leq i\leq n $. $ \mathscr{F} $ is a function class. Consider the random variable
\begin{equation*}
Z=\sup_{f\in \mathscr{F}}\frac{1}{n}\sum_{i=1}^{n}f(Z_i).
\end{equation*}
If $ \Vert f\Vert_{\infty}\leq b $ for all $ f\in\mathscr{F} $, we have the following inequality, for $ t>0 $,
\begin{equation}\label{inequality:empirical}
\Pb(Z\geq c_0\Eb[Z]+c_1 \tilde\sigma \sqrt{t}+c_2bt)\leq \exp\left(-nt\right)
\end{equation}
where $ \tilde\sigma^2=\sup_{f\in \mathscr{F}} \frac{1}{n}\sum_{i=1}^{n} \Eb[f^2(Z_i)] $ and $c_0, c_1, c_2$ are some universal constants.
\end{lemma}%end of lemma CI for EPs
For  $(\alpha,\rho) \in \mathbb{R}^2$, we define the space:
\begin{equation*}
S(\alpha, \rho)=\{A: 
\Vert A\Vf^2=1, \Vert A\Vert_*\leq \rho, \Vert A\Vert_{\infty}\leq \alpha\}.
\end{equation*}
Denote
\begin{equation*}
 F_A(X_i,\xi_i)=G_{\frac{\tau}{4s}}(\la X_i, A\ra)\ind_{\chi_i}.   
\end{equation*}
for  $A\in S(\alpha,\rho)$. 
Recall $ f(A)=\Eb_{X_i,\xi_i}[F_A(X_i,\xi_i)] $.
Since $x^2\geq x^2\ind_{|x|\leq 2R}\geq G_{R}(x)$,  for $A\in S(\alpha,\rho)$,
\begin{align*}
\quad~ \Eb[F^2_A(X_i,\xi_i)]-f^2(A)
&\leq \Eb[F^2_A(X_i,\xi_i)]\\
&\leq \Eb[G^2_{\tau/4s}(\la X_i,A\ra)]\\
&\leq \alpha^2 \V A\Vf^2\\
&\leq \alpha^2,
\end{align*}
which implies $\tilde \sigma\leq \alpha$ in (\ref{inequality:empirical}).
Consider
\begin{equation*}
Z(\alpha,\rho)=\sup_{A\in S(\alpha,\rho)}\left|\dfrac{1}{n}\sumn F_A(X_i,\xi_i)-f(A)\right|
\end{equation*}
We have
\begin{align*}
\Eb[Z(\alpha,\rho)]&\stackrel{\text{i}}{\leq} \Eb\sup_{A\in\Omega(\alpha,\rho)}\left|  \dfrac{1}{n}\sumn \epsilon_iF_A(X_i,\xi_i)\right|\\
&\stackrel{\text{ii}}{\leq} 2\alpha \Eb\sup_{A\in\Omega(\alpha,\rho)}\left|\dfrac{1}{n}\sumn\epsilon_{i}\langle X_{i}\ind_{\chi_i},A\rangle\right|\\
&\stackrel{\text{iii}}{\leq} 4\alpha \rho \Eb\left\Vert \dfrac{1}{n}\sumn \epsilon_iX_i\ind_{\chi_i}\right\Vert.
\end{align*}
In step (i) we use symmetrization technique. 
By Property (ii) of $G_R(\cdot)$ and $ \Vert A\Vert_{\infty}\leq \alpha $, $ G_{\frac{\tau}{4s}}(\cdot) $ is Lipschitz with Lipschitz constant less than $ 2\alpha $.
Then, in step (ii) we apply  Talagrand-Ledoux contraction inequality.
Step (iii) is due to $ |\la A, B\ra|\leq \Vert A\Vert \Vert B\Vert_*  $.

We take $t=\frac{2\mu \log d}{n} $ in (\ref{inequality:empirical}). Since $\tilde \sigma \leq \alpha$ and $\|F_A(X_i,\xi_i)\|_\infty\leq \alpha^2 $, with probability at least  $ 1-\exp(-2\mu \log d)$, it holds
\begin{equation}\label{inequality: 1 RSC Huber 1}
Z(\alpha,\rho)\leq 4c_0 \alpha\rho \Eb[\|\Rc\|]
+\sqrt{2}c_1\alpha\sqrt{\frac{\mu\log d}{n}}
+4c_2\frac{\mu\alpha^2\log d}{n}.
\end{equation}
Note that under Assumption \ref{assumption:X_i},
\begin{equation*}
\|A\|_F\leq \sqrt{\mu m_1m_2}\|A\|_{w(F)}=\sqrt{\mu m_1m_2}.
\end{equation*}
and 
\begin{equation*}
\|A\|_F \geq  \|A\|_{w(F)}=1. \end{equation*}
Then, we have
\begin{equation*}
\V A\|_*\leq \sqrt{m}\V A\V_F\leq \sqrt{\mu d^3}
\end{equation*}
and 
\begin{equation*}
1\leq \|A\|_F\leq \|A\|_*.    
\end{equation*}
Without loss of generality, we can assume $\|A\|_*\in[1, \sqrt{\mu d^3}]$.
Similarly, we have
\begin{equation*}
  \|A\|_\infty\leq \|A\|_F\leq \sqrt{\mu m_1m_2}  
\end{equation*}
 and
 \begin{equation}
   \|A\|_{\infty}\geq \frac{\|A\|_F}{\sqrt{m_1m_2}}\geq \frac{1}{d}.  \label{ld:max norm}
 \end{equation}
Consider
\begin{align*}
\tilde S(k,l)=\{A;\|A\|_{w(F)}=1, 2^{k-1}\leq \V A\V_{\infty}\leq 2^{k},~\text{and}~2^{l-1}\leq \V A\|_*\leq2^{l}\}
\end{align*}
Every matrix $A$ with $\|A\|_{w(F)}=1$ belongs to some $ \tilde S(k,l) $ with $k=\lfloor \log_2 \frac{1}{d}\rfloor, \ldots$, $\lceil \frac{1}{2}\log_2 (\mu d^2 )\rceil$ and $ l= 1,\ldots, \lceil \log_2 \frac{1}{2}(\mu d^3)\rceil $. Thus the total number of the pairs $(k,l)$ is less than $\Lambda^2= 9\log_2^2 (\mu d) $. 

Denote the events by $ \mathcal{E} $ that a matrix $ A $ violates the following relation
\begin{align}\label{P1.1}
\notag&\left| \dfrac{1}{n}\sumn F_A(X_i,\xi_i)-f(A)\right|\\
&\leq 16c_0\|A\|_{\infty}\|A\|_*\Eb[\|\Rc\|]
+2\sqrt{2}c_1\|A\|_\infty\sqrt{\frac{\mu \log d}{n}}
+16c_2\frac{\mu\V A\V_{\infty}^2\log d}{n}
\end{align}
and denote $ \mathcal{E}_{k,l} $ the event that there exists a matrix in $ \tilde S(k,l) $ violating the following relation
\begin{align*}
\left| \dfrac{1}{n}\sumn F_A(X_i,\xi_i)-f(A)\right|
\leq 4c_0 2^k 2^l \Eb[\|\Rc\|]
+\sqrt{2}c_1 2^k \sqrt{\frac{\mu \log d}{n}}
+4c_2\frac{\mu\V A\V_{\infty}^2\log d}{n}.
\end{align*}
$ \mathcal{E} $ is contained in the union $ \cup_{k,l}\mathcal{E}_{k,l} $.
Indeed, letting $ \alpha=2^{k} $ and $ \rho=2^{l} $, for a matrix $ A $ in $ \tilde S(k,l) $, we have
\begin{align*}
&\left| \dfrac{1}{n}\sumn F_A(X_i,\xi_i)-f(A)\right|\\
&\geq 16c_0 \|A\|_{\infty}\|A\|_*\Eb[\|\Rc\|]
+2\sqrt{2}c_1 \|A\|_\infty\sqrt{\frac{\mu \log  d}{n}}
+16c_2\frac{\mu\V A\V_{\infty}^2\log d}{n}\\
&\geq 16c_02^{k-1}2^{l-1}\Eb[\|\Rc\|]
+2\sqrt{2}c_1 2^{k-1}\sqrt{\frac{\mu 
\log d}{n}}
+16c_2 2^{2(k-1)}\frac{\mu \log d}{n}\\
&=4c_0\alpha\rho
+\sqrt{2}\alpha\sqrt{\frac{\mu \log d}{n}}
+4c_2 \alpha^2\frac{\mu\log d}{n}.
\end{align*}
Thus, by (\ref{inequality: 1 RSC Huber 1}), it holds that
\begin{equation*}
\Pb(\mathcal{E})\leq \Pb(\cup_{k,l}\mathcal{E}_{k,l})\leq \Lambda^2 \exp(-2\mu\log d)\leq \frac{9\log_2^2 \mu d}{d^{\mu}} \cdot\frac{1}{d^{\mu}}\leq \frac{1}{d}
\end{equation*}
for sufficiently large $d$.
\end{proof}%end of proof of RSC for Huber loss

%end of subsection: proof for heavy tailed  noise matrix completion

\subsection{Proof of Theorem \ref{theorem: new lasso}}\label{Pf2.3}

%Theorem \ref{theorem: new lasso} is a conclusion of Theorem \ref{theorem: new heavy tail}. 
%For $\tau_1>\tau_2\geq 0$, define 
%\begin{align*}
%&L(A)=\sumi (Y_i-\la X_i,A\ra)^2, \\
%    &L_{\tau_1}(A)=\sumi l_{\tau}(Y_i-\la X_i,A\ra), \\
%     &L_{\tau_2}(A)=\sumi l_{\tau_2}(Y_i-\la X_i,A\ra).
%\end{align*}
%Note 
%\begin{equation*}
%    \{\}
%\end{equation*}
%denote the corresponding estimators from (\ref{estimator: heavy tail}) by $\hat{A}_{H,\tau_1}, \hat{A}_{H,\tau_2}$.
%Note 
%\begin{equation*}
%\{\Ah_L=\hat{A}_{H,\tau_2}\} \subseteq   \{\Ah_L=\hat{A}_{H,\tau_1}\}.
%\end{equation*}
%Let $L(A)=$. Think of this question tomorrow.
From Appendix A in \cite{klopp2014noisy}, the following inequalities hold 
\begin{align}
\label{PF2.3eq1}  &\sumi  \la X_i, \Ah-A_0\ra^2\leq \frac{5}{3}\lambda\sqrt{r}\|\Ah-A_0\|_F,\\
\label{PF2.3eq2} &\|\Ah-A_0\|_*\leq \sqrt{72r}\|\Ah-A_0\|_F,
\end{align}
given
\begin{equation*}
     \lambda\geq 3\left\|\sxix\right\|.
\end{equation*}

Recall the definition
\begin{equation*}
    \|A\|_{w(F)}=\sqrt{\sum_{i=1}^{m_1}\sum_{j=1}^{m_2}P_{ij}a^2_{ij}}.
\end{equation*} 
Now we state a technical lemma.
\begin{lemma}\label{lemma:rsc L2}

For all matrices $A$, uniformly with probability at least $1-\frac{1}{d}$, we have
\begin{equation*}
  \dfrac{1}{n}\sumn \la X_i,A\ra^2\geq \frac{4}5{}\|A\|^2_{w(F)}-C\lb\|A\|_\infty\|A\|_*\Eb[\|\Rc\|]+\frac{\mu\V A\V_{\infty}^2\log d}{n} \rb,   
\end{equation*}
where $C$ is a universal  constant. Here  $\Rc=\sumi \epsilon_i X_i $, and $\{\epsilon_i\}_{i=1}^n$ are symmetric Bernoulli random variables independent of $\{(X_i,Y_i)\}_{i=1}^n$.

\end{lemma}
By Lemma \ref{lemma:rsc L2},  the following holds with probability at least $1-\frac{1}{d}$, 
\begin{align}
\notag &\dfrac{1}{n}\sumn \la X_i,\Ah-A_0\ra^2\geq \frac{4}{5}\|\Ah-A_0\|^2_{w(F)}\\
&- C\lb\|\Ah-A_0\|_\infty\|\hat{A}-A_0\|_*\Eb[\|\Rc\|]
+\frac{\mu\V \Ah-A_0\V_{\infty}^2\log d}{n} \rb. \label{PF2.3eq3}
\end{align}

Since $\|\Ah-A_0\|_\infty\leq 2a$ and (\ref{PF2.3eq2}) holds, we have
\begin{align*}
\lambda\sqrt{r}\|\Ah-A_0\|_F&\leq 5\mu\lambda^2 rm_1m_2 +\frac{\|\Ah-A_0\|^2_{F}}{5\mu m_1m_2}\\
&\leq 5\mu\lambda^2 rm_1m_2+\frac{1}{5}\|\Ah-A_0\|^2_{w(F)},\\
 \frac{\mu  \|\Ah-A_0\|^2_\infty \log d}{n}
&\leq \frac{4\mu a^2\log d}{n},\\   
\|\Ah-A_0\|_\infty\|\Ah-A_0\|_*\Eb[\|\Rc\|]
&\leq \sqrt{72}\|\Ah-A_0\|_\infty \sqrt{r}\,\|\Ah-A_0\|_F\Eb[\|\Rc\|] \\
&\leq  360C\mu a^2 r m_1 m_2\, (\Eb[\|\Rc\|])^2 
    + \frac{1}{5C}\|\Ah-A_0\|^2_{w(F)},
\end{align*}
where  we use the fact $\frac{\|A\|_F^2}{\mu m_1m_2}\leq \|A\|_{w(F)}^2$ and Cauchy-Schwarz inequality. 
Combining the above three bounds with (\ref{PF2.3eq1}) and (\ref{PF2.3eq3}), the following holds with probability at least $1-\frac{1}{d}$
\begin{equation}\label{P2.3}
\|\Ah-A_0\|^2_{w(F)}\lesssim   \mu rm_1 m_2 \lb  \lambda^2 + a^2(\Eb[\|\Rc\|])^2 \rb+\frac{\mu a^2 \log d}{n}.
\end{equation}
By (\ref{inequality:nasym-spectral}) in Lemma \ref{lemma: sharp spectral},  given $n\geq m\log^{4}d\cdot(\log d+\log n)$, the following holds with probability at least $1-\frac{1}{d}$,
\begin{equation*}
    \lnorm\sxix\rnorm\leq C_1\sigma\sqrt{\frac{1}{nm}},
\end{equation*}
where $C_1$ is a constant dependent on $\{ L_i\}_{i=1}^3$.
Then, we can take $\lambda=3C_1 \sigma\sqrt{\frac{1}{nm}}$ so that $\lambda\geq 3\|\sxix\|$.
By (\ref{inequality:E-spectral}) in Lemma \ref{lemma: sharp spectral}, given $n\geq m\log^4 d$, we have
\begin{equation*}
\Eb[\lnorm\Rc\rnorm]\lesssim\sqrt{\frac{1}{nm}}.
\end{equation*}
Plugging  $\lambda$ and the bound for $\Eb[\|\Rc\|]$ into (\ref{P2.3}),     we  obtain
\begin{equation*}
 \frac{\|\Ah-A_0\|^2_F}{m_1m_2}\lesssim\frac{\mu^2 \max(a^2,\sigma^2)rM}{n}+\frac{\mu^2a^2\log d}{n},   
\end{equation*}
which leads to the statement of Theorem \ref{theorem: new lasso}.
%end of the proof of new lasso.

\subsection{Proof of restricted strong convexity for the quadratic loss function}
\label{section:rsc quadratic}

%Proof of RSC for quadratic loss
\begin{proof}[\textbf{Proof of Lemma \ref{lemma:rsc L2}}]
%Note that Lemma \ref{lemma: RSC Huber 1} holds for all $(\tau,s)$. 
%Let $\tau\rightarrow\infty$, then we obtain $G_{\frac{\tau}{4s}}(x)\rightarrow x^2$ and $f(A)\rightarrow \|A\|^2_{w(F)}$.
%By bounded convergence theorem, (\ref{inequality:rsc L2}) holds with probability ar least $1-\frac{1}{d}$.
%\begin{equation*}
%    \left| \dfrac{1}{n}\sumn G_{\frac{\tau}{4s}}(\la X_i,A\ra )-f(A)\right|
%\leq 8c_0\|A\|_\infty\|A\|_*\Eb[\|\Rc\|]
%+2\sqrt{2}c_1\|A\|_\infty\sqrt{\frac{\mu \log d}{n}}
%+8c_2\frac{\mu\V A\V_{\infty}^2\log d}{n} 
%\end{equation*}
At first, let us consider $\|A\|_{w(F)}=1$.
Denote 
\begin{align*}
&Z_\tau(A)=  \left| \dfrac{1}{n}\sumn G_{\frac{\tau}{4s}}(\la X_i,A\ra )\ind_{\chi_i}-f_\tau(A)\right|,\\
&Z=\left| \dfrac{1}{n}\sumn \la X_i,A\ra^2-1\right|,\\
&\kappa_{\tau}(A)=8c_0\|A\|_\infty\|A\|_*\Eb[\|\Rc_\tau\|]
+2\sqrt{2}c_1\|A\|_\infty\sqrt{\frac{\mu \log d}{n}}
+8c_2\frac{\mu\V A\V_{\infty}^2\log d}{n},\\
&\kappa(A)= 8c_0\|A\|_\infty\|A\|_*\Eb[\|\Rc\|]
+2\sqrt{2}c_1\|A\|_\infty\sqrt{\frac{\mu \log d}{n}}
+8c_2\frac{\mu\V A\V_{\infty}^2\log d}{n},
\end{align*}
where 
\begin{equation*}
    \Rc_{\tau}=\sumi \epsilon_i X_i \ind(|\xi_i|\leq \tau/2),
\end{equation*}
and
\begin{equation*}
    	f_\tau(A):=\Eb_{X_i,\xi_i} [G_{\frac{\tau}{4s}} (X_i,A)\ind(|\xi_i|\leq \tau/2)].
\end{equation*}
Then, we consider the variables
\begin{align*}
&Y_\tau= \sup_{\|A\|_{w(F)}=1} \frac{Z_\tau(A)}{\kappa_\tau(A)},\\
&Y= \sup_{\|A\|_{w(F)}=1} \frac{Z}{\kappa(A)}.
\end{align*}
By Lemma \ref{lemma: RSC Huber}, we have, for $\tau>0$,
\begin{equation*}
 \Pb(Y_\tau\leq 1)\geq 1-\frac{1}{d}.   
\end{equation*}
Given $\|A\|_{w(F)}=1$, we have that $\kappa_\tau(A)$ is uniformly lower bounded away from zero by (\ref{ld:max norm}):
\begin{align*}
   \kappa_\tau(A)&\geq  2\sqrt{2}c_1\|A\|_\infty\sqrt{\frac{\mu \log d}{n}}
+8c_2\frac{\mu\V A\V_{\infty}^2\log d}{n}\\
&\geq \frac{2\sqrt{2}}{d}\cdot\sqrt{\frac{\mu\log d}{n}}+\frac{8c_2\mu\log d}{d^2n}.
\end{align*}
Since $\|\Rc_\tau\|\leq n$, by bounded convergence theorem, we obtain that  as $\tau$ tends to infinity, $\Eb[\|\Rc_\tau\|]$ converges to $\Eb[\|\Rc\|]$.
Note 
\begin{align*}
  \left| \dfrac{1}{n}\sumn G_{\frac{\tau}{4s}}(\la X_i,A\ra )\ind_{\chi_i} -\dfrac{1}{n}\sumn \la X_i,A\ra^2\right|&\leq C_1sa^2\sumi
\ind\left(|\xi_i|>\frac{\tau}{2}\right),\\
\left|f_\tau(A)- 1 \right|&\leq  C_2sa^2/\tau,
\end{align*}
where $C_1$ and $C_2$ are constants that are not dependent on $\tau$.
$Y_\tau$ converges  to $Y$ in probability and thus in distribution. There exist constants y dense in $\Rb$ such that
\begin{align*}
 \lim_{\tau\rightarrow\infty}\Pb(Y_\tau\leq y)=\Pb(Y\leq y). 
\end{align*}
For $y\geq 1$, by Lemma \ref{lemma: RSC Huber 1}, it holds that
\begin{equation*}
 \Pb(Y_\tau\leq y)\geq1-\frac{1}{d}   
\end{equation*}
for all $\tau$. 
Thus, there exists some $y\geq1$ such that
\begin{equation*}
    \Pb(Y\leq y)\geq1-\frac{1}{d}.
\end{equation*}

Then we have the following result. 
For matrices $A$ such that $\|A\|_{w(F)}=1$,  there is a constant $C$ dependent  on $ \{c_i\}_{i=0}^2 $ such that 
\begin{align*}
\notag&\left| \dfrac{1}{n}\sumn \la X_i,A\ra^2-1\right|\\
&\leq C\lb\|A\|_\infty\|A\|_*\Eb[\|\Rc\|]
+\|A\|_\infty\sqrt{\frac{\mu \log d}{n}}
+\frac{\mu\V A\V_{\infty}^2\log d}{n} \rb, 
\end{align*}
uniformly with probability at least $ 1-\frac{1}{d} $. 
For the general case, we  apply the above inequality to $A/\|A\|_{w(F)}$ and multiply $\|A\|^2_{w(F)}$ on both side. Then, we  obtain
\begin{align*}
    \notag&\left| \dfrac{1}{n}\sumn \la X_i,A\ra^2 -  \|A\|^2_{w(F)}  \right|\\
&\leq C\lb\|A\|_\infty\|A\|_*\Eb[\|\Rc\|]
+\|A\|_{w(F)}\|A\|_\infty\sqrt{\frac{\mu \log d}{n}}
+\frac{\mu\V A\V_{\infty}^2\log d}{n} \rb,
\end{align*}
which implies
\begin{align}
\notag   \dfrac{1}{n}\sumn \la X_i,A\ra^2\geq
    &\|A\|^2_{w(F)}
   - C\Bigg(\|A\|_\infty\|A\|_*\Eb[\|\Rc\|]\\
   &+\|A\|_{w(F)}\|A\|_\infty\sqrt{\frac{\mu \log d}{n}}+\frac{\mu\V A\V_{\infty}^2\log d}{n} \Bigg). \label{inequality:rsc L2}
\end{align}
For deviation error term $\|A\|_{w(F)}\|A\|_\infty\sqrt{\frac{\mu \log d}{n}}$, we apply Cauchy-Schwartz inequality 
\begin{equation*}
     C\|A\|_{w(F)}\|A\|_\infty\sqrt{\frac{\mu \log d}{n}}\leq \frac{1}{5}\|A\|^2_{w(F)}+5C^2\|A\|^2_{\infty}\frac{\mu \log d}{n}.
\end{equation*}
Plugging the above bound into (\ref{inequality:rsc L2}) leads to the statement of Lemma \ref{lemma:rsc L2}.

\end{proof} %end of proof for RSC quadratic loss

\subsection{Proof of Theorem \ref{theorem: new sq lasso}}
In this section, we set
\begin{equation*}
    \xi=(\xi_1/\sqrt{n},\ldots,\xi_n/\sqrt{n}).
\end{equation*}
From Appendix D in \cite{klopp2014noisy}, given
\begin{equation}\label{Psq1} %Proof of sq lasoo, equation 1
   \lambda\geq \frac{3\|\sxix\|}{\|\xi\|_2}, 
\end{equation} 
we have
\begin{equation}\label{inequality: lr As}% nearly low rank property of A_S
    \|\Ah_S-A_0\|_*\leq \sqrt{9r}\|\Ah_S-A_0\|_F.
\end{equation}
Again, from Appendix D in \cite{klopp2014noisy}, we have
\begin{equation*}
  \sumi\la X_i, \As-A_0\ra^2 \leq \frac{14}{3} \lambda \|\xi\|_2 \sqrt{2 r} \dasf + \frac{1}{4}\|\As-A_0\|^2_{w(F)}, 
\end{equation*}
under  condition \( 4 \mu m_1 m_2 \lambda^2 r \leq 1/4 \).

Recall the definition
\begin{equation*}
    \|A\|_{w(F)}=\sqrt{\sum_{i=1}^{m_1}\sum_{j=1}^{m_2}P_{ij}a^2_{ij}}.
\end{equation*} 

By Lemma \ref{lemma:rsc L2}, we obtain

\begin{align*}
\frac{4}{5}\daswf^2&\leq C(\|\As-A_0\|_\infty\|\As-A_0\|_*\Eb[\|\Rc\|]\\
&+\frac{\mu \|\As-A_0\|^2_\infty\log d}{n})
+\frac{14}{3} \lambda \|\xi\|_2 \sqrt{2 r} \dasf \\
&+ \frac{1}{4}\|\As-A_0\|^2_{w(F)}.
\end{align*}

Using $\|\As-A_0\|_{\infty}\leq 2a$ and (\ref{inequality: lr As}), 
and repeating the argument in the proof of Theorem \ref{theorem: new lasso}, we derive
\begin{align}\label{P3.1}
\daswf^2\lesssim  \mu m_1 m_2  r\lb \|\xi\|^2_2\lambda^2 + a^2  (\Eb[\|\Rc\|])^2 \rb  +\frac{\mu a^2\log d}{n},
\end{align}
where $\Rc=\sumi \epsilon_iX_i$ and $\{\epsilon_i\}_{i=1}^n$ are i.i.d. symmetric Bernoulli random variables independent of $\{(X_i,Y_i)\}_{i=1}^n$.
By Theorem 3.1.1 in \cite{vershynin2018high}, 
we have
\begin{equation*}
\Pb\lbr\left|\lnorm \xi\rnorm_2-\sigma\right|\geq L_1\sigma t \rbr\leq 2\exp(-C_1nt^2),    
\end{equation*}
under Assumption \ref{assumption:noise-subgaussian}.
Taking $t=0.1/L_1$, we obtain that with probability at least $1-2\exp(-C_1n/(100L^2_1))$, the following holds
\begin{equation*}
 0.9\sigma\leq     \|\xi\|_2 \leq 1.1\sigma.
\end{equation*}
By  (\ref{inequality:nasym-spectral}) in Lemma \ref{lemma: sharp spectral}, given $n\geq m\log^{4}d\cdot(\log d+\log n)$,  the following holds with probability at least $1-\frac{1}{d}$
\begin{equation*}
    \lnorm\sxix\rnorm \leq C_2\sigma\sqrt{\frac{1}{nm}},
\end{equation*}
for some constant $C_2$. 
Thus, with probability at least $1-\frac{1}{d}-2\exp(-C_1n/(100L^2_1))$, 
\begin{equation*}
   \frac{3\|\sxix\|}{\|\xi\|_2}\leq 3.6C_2\sqrt{\frac{1}{nm}}.
\end{equation*}
Thus, we can take $\lambda=3.6C_2\sqrt{\frac{1}{nm}}$ so that (\ref{Psq1}) is satisfied.
By (\ref{inequality:E-spectral}) in Lemma \ref{lemma: sharp spectral}, we have
\begin{equation*}
\Eb\left[\lnorm\sumi\epsilon_iX_i\rnorm\right]\leq C_3\sqrt{\frac{1}{nm}},
\end{equation*}
for some constant $C_3$.
Plugging $\lambda$ and the above the bounds of  $\Eb\left[\lnorm\sumi\epsilon_iX_i\rnorm\right]$ into (\ref{P3.1}), we obtain that with probability at least $1-\frac{2}{d}-2\exp(-C_1n/(100L_1^2))$, the following holds
\begin{equation}\label{AS:bd1}
    \frac{\|\hat{A}_S - A_0\|_2^2}{m_1 m_2} 
\lesssim  \frac{\mu^2\max(a^2,\sigma^2)rM}{n}+\frac{\mu^2a^2\log d}{n},
\end{equation}
where the constants are dependent on $\{L_i\}_{i=1}^3$.
This leads to the error bound in Theorem \ref{theorem: new sq lasso} as the second term above on RHS is less than the first term.

The requirement $ 4 \mu m_1 m_2 \lambda^2 r \leq 1/4$ holds when
\begin{equation*}
 n\geq 576C^2_2\mu M r.
\end{equation*}
For a sufficiently large universal constant $C_4$, given $n\geq C_4L_1^2\log d$, we have
\begin{equation*}
   2\exp(-C_1n/(100L^2_1))\leq \frac{1}{d}. 
\end{equation*}
Thus, (\ref{AS:bd1}) holds with probability at least $1-\frac{3}{d}$ given
\begin{equation*}
  n\geq C_5\mu M r,  
\end{equation*}
where $C_5$ is a constant depending on $\{L_i\}_{i=1}^3$.

%End of the poof of Theorem 2.5.

\subsection{Sharp bounds for spectral norms}
\label{section:spn}
 We first present the  advanced inequalities introduced in \cite{brailovskaya2024universality}. Let $ Q=\sum_{i=1}^nQ_i \in\mathbb{R}^{m_1\times m_2}$ be a random matrix where $ Q_i $ are independent random matrices with zero mean. 
We define
\begin{align*}
	\gamma(Q) &=\left(\max (\Vert \Eb [QQ^T]\Vert, \Vert \Eb[Q^TQ]\Vert)\right)^{1/2},\\
	\gamma_{*}(Q) &=\sup _{\|y\|_2=\|z\|_2=1} \left(\Eb[(y^TQz)^2]\right)^{1/2},  \\
	g(Q) &=\|\operatorname{Cov}(Q)\|^{1/2}, \\
	  R(Q) & =  \lnorm\max_{1 \leq i \leq n} \|Q_i\|\rnorm_{L^\infty} ,
\end{align*}
where $ \cov(Q)\in \Rb^{m_1m_2\times m_1m_2} $ is the covariance matrix of all the entries of $ Q $ and $\|\cdot\|_{L^\infty}$ is the essential supremum of a random variable.
\begin{theorem}[Corollary 2.17 in \cite{brailovskaya2024universality}]\label{theorem:matrix-concentration}
With the above notations, there exist a universal constant $ C_1 $ such that for $t\geq 0$
	\begin{align}
\notag		\Pb( \Vert Q\Vert &\geq2\gamma(Q)+
		C_1(g(Q)^{1/2} \gamma(Q)^{1/2}\log^{3/4} d\\
        &+\gamma_{*}(Q) t^{1/2}+R(Q)^{1/3} \gamma(Q)^{2/3} t^{2/3}+R(Q) t) )\leq de^{-t}   \label{inequality: vanhandel non-asym}
	\end{align}
and a universal constant $ C_2 $ such that
\begin{align}
\notag \Eb[\V Q\V]\leq 2\gamma(Q)+C_2(g(Q)^{1/2} \gamma(Q)^{1/2}\log^{3/4} d\\
+R(Q)^{1/3} \gamma(Q)^{2/3} \log ^{2/3}d+R(Q)\log d). \label{inequality:  vanhandel E}
\end{align}
\end{theorem}
\begin{remark}
The  matrix inequalities in \cite{brailovskaya2024universality} are formulated  for $Q_i\in \mathbb{R}^{d\times d}$ (not necessarily self-adjoint), and can be extended to rectangular matrices in the above form. 
See Remark 2.1 in \cite{brailovskaya2024universality}.

Theorem \ref{theorem:matrix-concentration} requires $Q_i$ to have bounded spectral norms, which is too restrictive in the context of matrix completion.
To leverage Theorem \ref{theorem:matrix-concentration}, we need to truncate the noise variables.
\cite{brailovskaya2024universality} also provides concentration inequalities  for the unbounded case (See Theorem 2.8 in \cite{brailovskaya2024universality}) that is derived by truncation techniques.
However, applying Theorem 2.8 in \cite{brailovskaya2024universality} could result in worse sample size conditions as \cite{brailovskaya2024universality} considers more general conditions.
For matrix completion, we need derive  specially tailored concentration inequalities  from Theorem \ref{theorem:matrix-concentration}.
\end{remark}

To present assumptions on the noise, we introduce a class of random variables that have exponential tail probability decays. A random variable $ Y $ is called $ \psi_\alpha $ random variable if it has finite   $ \psi_\alpha $ norm: 
\begin{equation*}
\Vert Y\Vert_{\psi_\alpha}=\inf\left\{C>0; \Eb\exp\left(\left(\frac{|Y|}{C}\right)^\alpha\right)\leq 2\right\}.
\end{equation*}
In particular, when $\alpha=2$, $\|\cdot\|_{\psi_2}$ is called sub-gaussian norm and X is called sub-gaussian random variable.

\begin{lemma}\label{lemma: sharp spectral}
Let $\{\epsilon_i\}_{i=1}^n$ be symmetric Bernoulli random variables that are independent of $\{(X_i,Y_i)\}_{i=1}^n$.

(1) Assume the sampling matrices  $X_i$ satisfy Assumption \ref{assumption:X_i} and  $\xi_i$ are $\psi_\alpha$ random variables with $\Eb[\xi_i]=0$, $\Eb[\xi_i^2|X_i]\leq \sigma^2$ and have uniformly bounded $\psi_\alpha$ norms
\begin{equation*}
 \lnorm [\xi_i|X_i]\rnorm_{\psi_\alpha} \leq L_1\sigma. 
\end{equation*}
The following inequality holds with probability at least $1-\frac{1}{d}$,
\begin{equation}\label{inequality:nasym-spectral}
	\left\| \dfrac{1}{n}\sum_{i=1}^n \xi_i X_i\right\|\leq C\sigma\sqrt{\dfrac{1}{nm}},
\end{equation}
given $n\geq m \log^4 d \cdot (\log^{2/\alpha}n+\log^{2/\alpha}d)$. 
In particular, for sub-Gaussian noise, (\ref{inequality:nasym-spectral}) holds given $n\geq m \log^4 d \cdot (\log n+\log d)$.
Here $C$ is a constant dependent on $\{L_i\}_{i=1}^3$.

(2)Assume $\xi_i$ satisfy Assumption \ref{assumption: mean var} and are symmetric and $X_i$ satisfy Assumption \ref{assumption:X_i}.  
With $\tau =\frac{\max(\sigma,a)}{\log^2 d}\sqrt{\frac{n}{m}}$, the following inequality holds with probability at least $1-\frac{1}{d}$,
\begin{equation}\label{inequality:nasym-spectral-bounded}
 \left\| \dfrac{1}{n}\sum_{i=1}^n \phi_\tau(\xi_i) X_i\right\|\leq C\max(\sigma,a)\sqrt{\dfrac{1}{nm}}.
\end{equation}
where $C$ is dependent on $L_2$ and $L_3$.

(3) If $ n\geq m\log^4 d   $, we have the following result:
\begin{align}\label{inequality:E-spectral}
\Eb\left\Vert \dfrac{1}{n}\sumn \epsilon_iX_i\right\Vert\leq C\sqrt{\frac{1}{nm}},
\end{align}
where $C$ is dependent on $L_2$ and $L_3$.

(4) If $n\geq m\log^4 d$, we have the following inequality:
\begin{equation}\label{inequality: E-spectral heavy}
    \Eb\left\Vert \dfrac{1}{n}\sumn \epsilon_iX_i\ind_{\chi_i}\right\Vert\leq C\sqrt{\frac{1}{nm}},
\end{equation}
where $\chi_i=\{|\xi_i|\leq \tau/2\}$ and C is a constant dependent on $L_2$ and $L_3$.
\end{lemma}

\begin{proof}[\textbf{Proof of Lemma \ref{lemma: sharp spectral}}]

\textbf{Proof of Part (1)}. Since Theorem \ref{theorem:matrix-concentration} requires $\|Q_i\|$ to be bounded, we need to truncate $\xi_i$ to $\xi_i\textbf{1}(|\xi_i|\leq\tau)$ for some $\tau\geq 0$. 
Define $\bar{\xi}_i=\xi_i\textbf{1}(|\xi_i|\leq\tau)-\Eb[\xi_i\textbf{1}(|\xi_i|\leq\tau)|X_i]$. 
Note $\Eb[\bar{\xi}_i]=0$ and $\Eb[\bar{\xi}_i^2]\leq \sigma^2$.
We now bound the four parameters for $ Q=\frac{1}{n}\sum_{i=1}^{n}\xb_iX_i $.
Define the event
\begin{equation*}
\Ec=\{X_i=e_j(m_1)e_k(m_2)^T\}.  
\end{equation*}
\noindent
1. Bounding $\gamma(Q)$:\\
Note $ \Eb[\xb^2_iX_i X_i^{T}] $ is a diagonal matrix. The elements on the diagonal line are $\sum_{k=1}^{m_2} P_{jk}\Eb[\xb_i^2|\Ec], j=1,2,\ldots, m_1$, which are upper bounded by  $  \frac{ L_2\sigma^2}{m} $.
The same bound holds for $ \Eb[ \xb_i^2X_i^{T}X_i] $. 
Since 
\begin{equation*}
 \Eb[QQ^T]=\frac{1}{n^2}\sum_{i=1}^n \Eb[\xb_i^2 X_iX_i^T],   
\end{equation*}
we obtain $\| E[QQ^T]\|\leq \frac{L_2\sigma^2}{nm}$.
Similarly, $\|\Eb[Q^TQ]\|\leq \frac{L_2\sigma^2}{nm}$.
Therefore, we have
\begin{equation*}
    \gamma(Q)\leq \sqrt{\frac{L_2\sigma^2}{nm}}.
\end{equation*}
\noindent
2. Bounding $\gamma_*(Q)$:\\
For $ y\in\mathbb{R}^{m_1}, z\in\mathbb{R}^{m_2} $ with $ \Vert y\Vert_2=\Vert z\Vert_2=1 $, we have the following result:
\begin{align*}
	\Eb(y^T\xb_iX_iz)^2&=\sum_{j=1}^{m_1}\sum_{k=1}^{m_2}P_{jk}\Eb[\xb_i^2|\Ec] y^2_jz^2_k\\
 &\leq \sigma^2\sum_{j=1}^{m_1}\sum_{k=1}^{m_2}P_{jk} y^2_jz^2_k\\
 &\leq  \dfrac{L_3\sigma^2}{m\log^3d}.
\end{align*}
where we use $\sum_{j=1}^{m_1}\sum_{k=1}^{m_2}y^2_jz^2_k=\Vert y\Vert_2^2\cdot \Vert z\Vert_2^2=1$ and $\max P_{jk}\leq \frac{L_3}{m\log^3 d}$. Since $ \xb_iX_i $ are independent and mean-zero, $ \Eb(y^TQ z)^2=\dfrac{1}{n^2}\sum_{i=1}^{n}\Eb\left(y^T\xb_iX_iz\right)^2\leq \frac{L_3\sigma^2}{nm\log^3 d} $.
Hence, we conclude
\begin{equation*}
\gamma_*(Q)\leq \sqrt{\frac{L_3\sigma^2}{nm\log^3 d}}.    
\end{equation*}
\noindent
3.Bounding $g(Q)$: \\
Since $ \xb_iX_i $'s are independent, we have
\begin{equation*}
 \cov(Q)=\dfrac{1}{n^2}\sum_{i=1}^{n}\cov(\xb_iX_i).   
\end{equation*}
Define $ \ind_{(j,k)}=\textbf{1}(X_i=e_j(m_1)e_k(m_2)^T) $ (for simplicity of notation, we do not include a subscript i) and denote $ I(j,k)=\xb_i\ind_{(j,k)} $. Note the following facts
\begin{align*}
&\Eb[I(j,k)]=P_{jk}\Eb[\xb_i|\Ec]=0,\\
&I(j,k)I(l,h)=0 ~~\text{if}~~(j,k)\neq (l,h).
\end{align*} 
Thus $ \text{Cov}\big(I(j,k),I(l,h)\big)=\Eb[I(j,k)I(l,h)]-\Eb[I(j,k)]E[I(l,h)]=0 $ for $ (j,k)\neq (l,h) $. 
Then we obtain 
\begin{equation*}
	\text{Cov}(\xb_i X_i)\Big((j,k),(l,h)\Big)=
	\begin{cases}
		\Eb[\xb^2_i|\Ec]P_{jk} ~~\text{for}~~(j,k)=(l,h),\\
		0 ~~~\text{for}~~~(j,k)\neq(l,h).
	\end{cases}
\end{equation*}
Since $ P_{jk}\leq \frac{L_3}{m\log^3 d} $, it follows that $ \V \cov(\xb_iX_i)\V\leq \frac{L_3\sigma^2}{m\log^3 d} $.
Thus, we have
\begin{equation*}
 g(Q)\leq \sqrt{\frac{L_3\sigma^2}{nm\log^3 d}}.   
\end{equation*}
\noindent
\noindent
4. Bounding $R(Q)$:\\
Obviously, $ \|\xb_iX_i\|=|\xb_i| $, and $ \max_{1\leq i\leq n} |\xb_i|\leq 2\tau $. Thus, we have
\begin{equation*}
    R(Q)\leq\frac{2\tau}{n}.
\end{equation*}

Now, we  plug the four estimates into (\ref{inequality: vanhandel non-asym}). Choosing $ t=3\log d $, we get,  with probability at least $ 1-\frac{1}{d^2} $, the following holds
\begin{align}
\notag  \V Q \V&\leq 2\sqrt{\frac{\sigma^2}{nm}}\\
\notag  &+C\Bigg(\left(\frac{\sigma^2}{nm}\right)^{1/2}\cdot\left(\dfrac{\log^{3/4}d}{\log^{3/4} d}\right)   +\left(\frac{\sigma^2}{nm}\right)^{1/2}\cdot \left( \dfrac{\log^{1/2} d}{\log^{3/2} d}\right)\\
\label{inequality: Q1}  &+\left(\frac{\log^{2} d \cdot\tau}{n}\right)^{1/3}\cdot\left(\frac{\sigma^2}{nm}\right)^{1/3}
+\dfrac{\log d\cdot\tau}{n}\Bigg),
\end{align}
where $C$ is dependent on $L_2$ and $ L_3$.
The quantities in the first line and the second line of RHS of (\ref{inequality: Q1}) are bounded by $C\sqrt{\frac{\sigma^2}{nm}}$.
When $\xi_i$ are $\psi_\alpha$ random variables, we set $\tau=C_\alpha \sigma(\log^{1/\alpha}n+ \log^{1/\alpha} d)$. 
By (\ref{inequality: max alpha}), with a properly chosen $C_\alpha$ , we have
\begin{equation*}
\Pb\left(\max_{1\leq i\leq n}|\xi_i|\geq C_\alpha\sigma\left( \log^{1/\alpha}n+ \log^{1/\alpha} d\right)\right)\leq \frac{1}{d^2}.    
\end{equation*}
To ensure  the quantities in the last line of (\ref{inequality: Q1})  are  bounded by  $C\sqrt{\frac{\sigma^2}{nm}}$, we need
\begin{align*}
\left(\frac{\sigma\log^2 d\cdot(\log^{1/\alpha}n+ \log^{1/\alpha} d) }{n}\right)^{1/3}\cdot\left(\frac{\sigma^2}{nm}\right)^{1/3}\leq \sqrt{\dfrac{\sigma^2}{nm}}, \\
\dfrac{\sigma\log d\cdot(\log^{1/\alpha}n+ \log^{1/\alpha} d)}{n}\leq \sqrt{\dfrac{\sigma^2}{nm}}.
\end{align*}
The above conditions imply
\begin{align*}
    n&\geq    m \log^4 d \cdot (\log^{2/\alpha} n+\log^{2/\alpha}d),\\
    n&\geq  m\log^2 d\cdot (\log^{2/\alpha} n+\log^{2/\alpha}d).
\end{align*}
On $ \{\max_{1\leq i\leq n} |\xi_i|\leq \tau\}$,  we have
\begin{align*}
\left\|\dfrac{1}{n}\sum_{i=1}^n\xi_iX_i \right\|&=\left\| Q+\dfrac{1}{n}\sum_{i=1}^n\Eb[\xi_i\textbf{1}(|\xi_i|\leq\tau)]X_i\right\|\\&\leq \|Q\|+\left\|\dfrac{1}{n}\sum_{i=1}^n\Eb[\xi_i\textbf{1}(|\xi_i|>\tau)]X_i\right\|\\
&\leq \|Q\|+\sumi|\Eb[\xi_i\textbf{1}(|\xi_i|>\tau)|X_i]|,
\end{align*}
where we use $\Eb[\xi_i\ind(\xi_i\leq \tau)]=\Eb[\xi_i\ind(\xi_i> \tau)]$ since $\Eb[\xi_i]=0$.
Let  $\tau=C_\alpha \sigma(\log^{1/\alpha}n+ \log^{1/\alpha} d)$ and  random variable $Y$ satisfy   $\|Y\|_{\psi_\alpha}\leq L_1\sigma$. There exists  $C_\alpha$ dependent on $L_1$ such that the following holds
\begin{align*}
 \Eb[ Y\ind(|Y|>\tau)  ]&\leq \Eb[|Y|\ind(|Y|>\tau)]\\
 &\leq (\Eb[Y^2])^{1/2}P^{1/2}(|Y|>\tau)\\
 &\leq \sigma \exp\left(-2\max(\log^{1/\alpha }n, \log^{1/\alpha}d)^\alpha\right)\\
 &\leq \frac{\sigma}{n}\leq\sigma \sqrt{\frac{1}{nm}},
\end{align*}
given $n\geq m$.
Combining the results, we obtain
\begin{align*}
&\quad~\Pb\left( \left\| \sxix \right\|\geq C_\alpha\sigma \sqrt\frac{1}{nm}\right)\\
&\leq \Pb\left( \left\| \sxix\right\|\geq C_\alpha\sigma \sqrt\frac{1}{nm},\max_{1\leq i\leq n} |\xi_i|\leq \tau\right)\\
&+\Pb\left(\max_{1\leq i\leq n} |\xi_i|> \tau\right)  \\
&\leq \frac{1}{d},
\end{align*}
given $n\geq    m \log^4 d \cdot (\log^{2/\alpha} n+\log^{2/\alpha}d)$.
In particular, for sub-Gaussian noise,  we have 
\begin{equation*}
 \Pb\left(\lnorm \sxix \rnorm\leq C\sigma\sqrt\frac{1}{nm} \right)\geq 1-\frac{1}{d}   
\end{equation*}
given $n\geq \log^4d \cdot(\log n+ \log d)$. So far, we conclude (\ref{inequality:nasym-spectral}).\\

\noindent
\textbf{Proof of Part (2).}
To prove the dimension free bound for $\sumi \phi_\tau(\xi_i)X_i$, we define  $Q_i=\frac{1}{n}\phi_\tau(\xi_i)X_i$ and $Q=\sum_{i=1}^n Q_i$.
Since $\xi_i$ are symmetric, $\Eb[\phi_\tau(\xi_i)|X_i]=0$. By (\ref{inequality: truc var}),  we  obtain $\Eb[\phi_\tau(\xi_i)^2|X_i]\leq \sigma^2$, $|\phi_{\tau}(\xi_i)|\leq \tau$.
Therefore, the estimates we have obtained earlier for the four parameters still apply.
Then, (\ref{inequality: Q1}) holds for $Q=\frac{1}{n}\phi_\tau(\xi_i)X_i$. 
With $\tau =\frac{\max(\sigma,a)}{\log^2 d}\sqrt{\frac{n}{m}}$ and $n\geq m\log^4 d$, we have 
\begin{align*}
\left(\frac{\log^{2} d }{n} \cdot\frac{\max(\sigma,a)}{\log^2 d}\sqrt{\frac{n}{m}}\right)^{1/3}\cdot\left(\frac{\sigma^2}{nm}\right)^{1/3}\leq\max(\sigma,a) \sqrt{\dfrac{1}{nm}}, \\
\dfrac{\log d}{n}\frac{\max(\sigma,a)}{\log^2 d}\sqrt{\frac{n}{m}}\leq \max(\sigma,a)\sqrt{\dfrac{1}{nm}}.
\end{align*}
Hence, we obtain (\ref{inequality:nasym-spectral-bounded}).\\

\noindent
\textbf{Proof of Part (3).}
Similarly, let $Q=\frac{1}{n}\sum_{i=1}^n \epsilon_i X_i$. Since $\epsilon_i\leq 1$, we have $R(Q)=\frac{1}{n}$. 
We have already estimated the four parameters above. 
By (\ref{inequality:  vanhandel E}), it remains to ensure that
\begin{align*}
\left(\frac{\log^{2} d }{n}\right)^{1/3}\cdot\left(\frac{1}{nm}\right)^{1/3}\leq C\sqrt{\dfrac{1}{nm}}, \\
\dfrac{\log d}{n}\leq C\sqrt{\dfrac{1}{nm}},
\end{align*}
which leads to the condition $n\geq m\log^4 d.$
This completes the proof of (\ref{inequality:E-spectral}).\\

\noindent
\textbf{Proof of Part (4).}
In the following, we define $Q_i=\frac{1}{n}\epsilon_iX_i\ind_{\chi_i}$ and $Q=\sum_{i=1}^n Q_i$.\\
\noindent
1.Bounding $\gamma(Q)$:\\ 
Note $Q_iQ_i^T=\frac{1}{n^2}\indc X_iX_i^T$. $\Eb[\frac{1}{n^2}\indc X_iX_i^T]$ is a diagonal matrix with diagonal elements $\sum_{k=1}^{m_2}\Pb(\chi_i|\Ec)P_{jk}$.
Thus, we have $\|Q_iQ_i^T\|\leq \frac{L_2}{n^2m}$ and $\|QQ^T\|\leq \frac{L_2}{nm}$.
Similarly, we can obtain $\|QQ^T\|\leq \frac{L_2}{nm}$.
Therefore,
\begin{equation*}
    \gamma(Q)\leq \sqrt{\frac{L_2}{nm}}.
\end{equation*}
\noindent
2. Bounding $ \gamma_*(Q)$:\\
For $ y\in\mathbb{R}^{m_1}, z\in\mathbb{R}^{m_2} $ with $ \Vert y\Vert_2=\Vert z\Vert_2=1 $, we have the following result:
\begin{align*}
	\Eb[(y^T\xb_iX_iz)^2]&=\sum_{j,k}P_{jk}\Pb(\indc|\Ec) y^2_jz^2_k\\
 &\leq \sum_{j,k}P_{jk} y^2_jz^2_k\\
 &\leq  \dfrac{L_3}{m\log^3d}.
\end{align*}
where we use $\sum_{j=1}^{m_1}\sum_{k=1}^{m_2}y^2_jz^2_k=\Vert y\Vert_2^2\cdot \Vert z\Vert_2^2=1$ and $\max P_{jk}\leq \frac{L_3}{m\log^3 d}$. 
Since $ Q_i $ are independent and mean-zero, we have
\begin{equation*}
    \Eb(y^TQ z)^2=\dfrac{1}{n^2}\sum_{i=1}^{n}\Eb[\left(y^T\epsilon_iX_i\indc z\right)^2]\leq \frac{L_3}{nm\log^3 d}.
\end{equation*}
Thus, we have
\begin{equation*}
     \gamma_*(Q)\leq \sqrt{\frac{L_3\sigma^2}{nm\log^3 d}} .
\end{equation*}
\noindent
3. Bounding $g(Q)$: \\
Since $ Q_i $ are independent, we have
\begin{equation*}
 \cov(Q)=\dfrac{1}{n^2}\sum_{i=1}^{n}\cov(\epsilon_iX_i\indc).  
\end{equation*}
Define $ \ind_{(j,k)}=\textbf{1}(X_i=e_j(m_1)e_k(m_2)^T) $ (for simplicity of notation, we do not include a subscript i) and denote $ I(j,k)=\epsilon_i\ind_{(j,k)}\indc $. 
Note the following facts
\begin{align*}
&\Eb[I(j,k)]=\Eb[\epsilon_i]\Eb[\ind_{(j,k)}\indc]=0,\\
&I(j,k)I(l,h)=0 ~~\text{if}~~(j,k)\neq (l,h).
\end{align*} 
Thus $ \text{Cov}\big(I(j,k),I(l,h)\big)=\Eb[I(j,k)I(l,h)]-\Eb[I(j,k)]E[I(l,h)]=0 $ for $ (j,k)\neq (l,h) $. 
Then we obtain 
\begin{equation*}
	\text{Cov}(\epsilon_iX_i\ind_{\chi_i})\Big((j,k),(l,h)\Big)=
	\begin{cases}
		\Pb(\chi_i|\Ec)P_{jk} ~~\text{for}~~(j,k)=(l,h),\\
		0 ~~~\text{for}~~~(j,k)\neq(l,h).
	\end{cases}
\end{equation*}
Since $ P_{jk}\leq \frac{L_3}{m\log^3 d} $, $ \V \cov(\xb_iX_1)\V\leq \frac{L_3}{m\log^3 d} $.
Thus, it holds
\begin{equation*}
     g(Q)\leq \sqrt{\frac{L_3}{nm\log^3 d}}.
\end{equation*}
4. Bounding $R(Q)$:\\
Note $ \|Q_i\|=|\indc|/n $. Thus,
\begin{equation*}
   R(Q)\leq \frac{2}{n}. 
\end{equation*}
\noindent
Plugging  the four estimates for $\sumi \epsilon_iX_i\indc$ into (\ref{inequality:  vanhandel E}),  we conclude (\ref{inequality: E-spectral heavy}) given $n\geq m\log^4 d$.
\end{proof}
%end of proof for sharp bounds for spectral norms

\subsection{Concentration inequalities for the maximum of independent random variables}

Given $\|X\|_{\psi_\alpha}\leq \sigma$, it holds that for $t>0$ 
\begin{equation*}
    \Pb(|X|\geq \sigma t)\leq 2\exp(-t^{\alpha}). 
\end{equation*}
For a metric space $ (T,\rho) $, a sequence of its subsets $ \{T_m\}_{m\geq 0} $ is called an admissible sequence if $ |T_0|=1 $ and $ |T_m|\leq 2^{2^m} $ for $m\geq 1$. 
For $ 0<\alpha<\infty $, the functional $ \gamma_\alpha(T,\rho) $ is defined by 
\begin{equation*}
\gamma_\alpha(T,
\rho)=\inf_{\{T_m\}^{\infty}_{m=1}}\sup_{t\in T} \sum_{m=0}^{\infty} 2^{m/\alpha}\rho(t,T_m)
\end{equation*}
where the infimum is taken over all  admissible sequences. If a random process $ \{X_t, t\in T\} $ satisfies
\begin{equation*}
\Pb(|X_t-X_s|\geq t\rho(t,s))\leq 2\exp(-t^{\alpha}),
\end{equation*}
the following theorem holds.

\begin{theorem}[Theorem 3.2 in  \cite{dirksen2015tail}]
There exist constants $ C_{1,\alpha}, C_{2,\alpha} $ such that
\begin{equation}\label{in:dirksen}
\mathbb{P}\left(\sup _{t \in T}|X_t-X_{t_0}| \geq e^{1 / \alpha}\left(C_{1,\alpha}\gamma_\alpha(T, 
\rho)+t C_{2,\alpha} \Delta_\rho(T)\right)\right) \leq \exp \left(-t^\alpha / \alpha\right)
\end{equation}
where $ \Delta_{\rho}(T)=\sup_{t,s\in T} \rho(t,s) $, $t\geq 0$ and $C_{1,\alpha}, C_{2,\alpha}$ are dependent on $\alpha$.
\end{theorem}

Let $\{\xi_i\}_{i=1}^n $ be a sequence of independent $ \psi_\alpha $ random variables with $ \Vert \xi_i\Vert_{\psi_\alpha}\leq \sigma $.
By Lemma A.3 in  \cite{gotze2021concentration}, we have   
\begin{equation*}
	\Vert X+Y\Vert_{\psi_\alpha}\leq K_{\alpha}(\Vert X\Vert_{\psi_\alpha}+\Vert Y\Vert_{\psi_\alpha} ),
\end{equation*}
where $ K_{\alpha}=2^{1/\alpha} $ for $ \alpha\in (0,1) $ and $ K_{\alpha}=1 $ for $ \alpha\geq1 $.
We  define the trivial metric on $ T=\{1,2,\ldots,n\} $,
\begin{equation*}
\rho(t,s)=\begin{cases}
2K_{\alpha}\sigma, t\neq s,\\
0,t=s.
\end{cases}
\end{equation*}
With this metric, we  construct an admissible sequence $ \{T_n\} $. There exists $\tilde{m} $ such that $ 2^{2^{\tilde{m}}}\leq n\leq 2^{2^{\tilde{m}+1}} $.
For $ m\leq \tilde{m}+1 $, take $ T_m=\{1\}$. 
For $ m>\tilde{m}+1 $, take $ T_m=T $. 
Then, we obtain 
\begin{equation*}
\sup_{t\in T}\sum_{m=0}^{\infty}\rho(t,T_n)=\sup_{t\in T} \sum_{m=0}^{\tilde{m}+1} 2^{m/\alpha}\rho(t,T_m)=2K_{\alpha}\sigma \sum_{m=0}^{\tilde{m}+1} 2^{m/\alpha}=\sigma C_{\alpha}2^{\tilde{m}/\alpha},
\end{equation*}
where $C_\alpha$ is a constant  dependent only on $\alpha$. 
Noting that  $ 2^{\tilde{m}} \lesssim \log n $, we  obtain
\begin{equation*}
\gamma_{\alpha}(T,\rho)\leq C_{\alpha}\sigma (\log n)^{1/\alpha}.
\end{equation*}
Observe that 
\begin{equation*}
\sup_{i\in\{1,\dots,n\}}|\xi_i| \leq |\xi_1|+   \sup_{i\in\{1,\dots,n\}}|\xi_i-\xi_1|.
\end{equation*}
By (\ref{in:dirksen}) we get
\begin{align*}
 &\Pb\left(\sup_{i\in\{1,\cdots,n\}}|\xi_i|\geq e^{1/\alpha}\left(C_{1,\alpha}\sigma(\log n)^{1/\alpha}+ (C_{2,\alpha}+1) \sigma t\right)\right) \\
 \leq&\Pb\left(\sup_{i\in\{1,\cdots,n\}}|\xi_i-\xi_1|\geq e^{1/\alpha}\left(C_{1,\alpha}\sigma(\log n)^{1/\alpha}+ C_{2,\alpha} \sigma t\right)\right)+\Pb\left(|\xi_1|\geq e^{1/\alpha}\sigma t\right)\\
 \leq &\exp(-t^{\alpha}/\alpha)+\frac{2}{e}\exp(-t^\alpha/\alpha)\leq 2\exp(-t^\alpha/\alpha) ,
\end{align*}
where we use  $\Pb(X+Y\geq a+b)\leq \Pb(X\geq a)+\Pb(Y\geq b)$. 
Equivalently, the concentration inequality reads
\begin{equation}\label{inequality: max alpha}
 \Pb\left(\sup_{i\in\{1,\cdots,n\}}|\xi_i|\geq C_{1,\alpha} \sigma(\log n)^{1/\alpha}+ C_{2,\alpha} \sigma t^{1/\alpha}\right) \leq 2\exp(-t).  
\end{equation}
As a consequence, if $ \{\xi_i\}_{i=1}^n $ is a sequence of independent sub-Gaussian random variables satisfying $ \Vert \xi_i\Vert_{\psi_2}\leq \sigma $, there exist constants $ C_1, C_2 $ such that
\begin{equation}\label{inequality: max subgaussian }
\Pb\left(\sup_{i\in\{1,\ldots,n\}} |\xi_i|\geq C_1 \sigma\sqrt{\log n}+C_2\sigma\sqrt{t}\right)\leq 2\exp(-t).
\end{equation}

\begin{lemma}
Let $Y$ be a random variable with zero mean such that $\Eb[Y^2]\leq \sigma^2$. 
Then, for any $y$, 
\begin{equation}\label{inequality: truc var}
 \var(\phi_{y}(Y))\leq\sigma^2 .   
\end{equation}
\end{lemma}
\begin{proof}[\textbf{Proof}]
 Let $Y'$ be an independent copy of $Y$. 
 Since  $\phi_y(\cdot)$ is Lipchitz with Lipchitz constant $1$, we obtain
 \begin{equation*}
  2\var(\phi_{y}(Y)) \leq  \Eb[(\phi_{y}(Y)-\phi_{y}(Y'))^2]\leq \Eb[(Y'-Y)^2]=2\Eb [Y^2]. 
 \end{equation*}
Then, we conclude the lemma.
\end{proof}

\section{Conclusion}\label{section: discussion}

In this paper, we considered three common matrix completion settings: (1) the unknown matrix is of low rank and the noise is heavy tailed; (2) the unknown matrix is of low rank and the noise is sub-Gaussian with known variance; (3) the unknown matrix is of low rank and the noise is sub-Gaussian with unknown variance. We revisited three popular estimators in these settings and developed sharper upper bounds to establish their minimax rate optimality.

A key technical contribution lies in employing powerful matrix concentration inequalities introduced in \cite{brailovskaya2024universality} to successfully remove the dimension factor $\log d$ that has remained in the upper bounds prior to our work. As previously noted, the $\log d$ factor is ubiquitous in matrix completion problems under the \emph{sampling with replacement} setting. Our sharp analyses (e.g., Lemma~\ref{lemma: sharp spectral} establishes dimension-free results for a variety of random matrices) can be applied or adapted to improve other existing results on matrix completion \citep{koltchinskii2011nuclear,negahban2012restricted,klopp2017robust,chen2024dynamic}. The current paper studies optimal rates in terms of sample size and matrix dimensions. An important future research is to establish a full characterization of minimax rates with respect to other problem parameters including $\{\mu,a,\sigma\}$.

\begin{acks}[Acknowledgments]
The authors would like to thank the anonymous referees,  and the editor for their constructive comments that significantly improved  the quality of this paper.
\end{acks}

\bibliographystyle{imsart-number} % Style BST file (imsart-number.bst or imsart-nameyear.bst)
\bibliography{ref}       % Bibliography file (usually '*.bib')

@article{candes2010power,
  title={The power of convex relaxation: Near-optimal matrix completion},
  author={Cand{\`e}s, Emmanuel J and Tao, Terence},
  journal={IEEE transactions on information theory},
  volume={56},
  number={5},
  pages={2053--2080},
  year={2010},
  publisher={IEEE},
    MRNUMBER = {2723472}

}

@article{recht2011simpler,
  title={A simpler approach to matrix completion},
  author={Recht, Benjamin},
  journal={Journal of Machine Learning Research},
  volume={12},
  number={12},
  year={2011},
    MRNUMBER = {2877360}
}

@article{koltchinskii2011nuclear,
  title={Nuclear-norm penalization and optimal rates for noisy low-rank matrix completion},
  author={Koltchinskii, Vladimir and Lounici, Karim and Tsybakov, Alexandre B},
  year={2011},
pages={2302-2329},
journal={Annals of Statisitcs},
volume={39},
  MRNUMBER = {2906869},
number={5}
}

@article{klopp2014noisy,
author = {Olga Klopp},
title = {{Noisy low-rank matrix completion with general sampling distribution}},
volume = {20},
journal = {Bernoulli},
number = {1},
publisher = {Bernoulli Society for Mathematical Statistics and Probability},
pages = {282 -- 303},
  MRNUMBER = {3160583},
year = {2014},
}

@article{klopp2017robust,
  title={Robust matrix completion},
  author={Klopp, Olga and Lounici, Karim and Tsybakov, Alexandre B},
  journal={Probability Theory and Related Fields},
  volume={169},
  pages={523--564},
  year={2017},
    MRNUMBER = {3704775},
  publisher={Springer}
}

@article{minsker2018sub,
  title={Sub-Gaussian estimators of the mean of a random matrix with heavy-tailed entries},
  author={Minsker, Stanislav},
  journal={Annals of Statistics},
  volume={46},
  number={6A},
  pages={2871--2903},
  year={2018},
  MRNUMBER = {3851758},
  publisher={JSTOR}
}

@article{yu2024low,
  title={Low-rank matrix recovery under heavy-tailed errors},
  author={Yu, Myeonghun and Sun, Qiang and Zhou, Wen-Xin},
  journal={Bernoulli},
  volume={30},
  number={3},
  pages={2326--2345},
  year={2024},
    MRNUMBER = {4746610},
  publisher={Bernoulli Society for Mathematical Statistics and Probability}
}

@article{bandeira2023matrix,
  title={Matrix concentration inequalities and free probability},
  author={Bandeira, Afonso S and Boedihardjo, March T and van Handel, Ramon},
  journal={Inventiones mathematicae},
  volume={234},
  number={1},
  pages={419--487},
  year={2023},
    MRNUMBER = {4635836},
  publisher={Springer}
}

@article{negahban2012restricted,
  title={Restricted strong convexity and weighted matrix completion: Optimal bounds with noise},
  author={Negahban, Sahand and Wainwright, Martin J},
  journal={The Journal of Machine Learning Research},
  volume={13},
  number={1},
  pages={1665--1697},
  year={2012},
    MRNUMBER = {2930649},
  publisher={JMLR. org}
}

@book{wainwright2019high,
  title={High-dimensional statistics: A non-asymptotic viewpoint},
  author={Wainwright, Martin J},
  volume={48},
  year={2019},
     MRNUMBER = {3967104},
  publisher={Cambridge university press}
}

@article{fan2021shrinkage,
  title={A shrinkage principle for heavy-tailed data: High-dimensional robust low-rank matrix recovery},
  author={Fan, Jianqing and Wang, Weichen and Zhu, Ziwei},
  journal={Annals of statistics},
  volume={49},
  number={3},
  pages={1239},
  year={2021},
    MRNUMBER = {4298863},
  publisher={NIH Public Access}
}

@article{elsener2018robust,
  title={Robust low-rank matrix estimation},
  author={Elsener, Andreas and van de Geer, Sara},
  journal={The Annals of Statistics},
  volume={46},
  number={6B},
  pages={3481--3509},
  year={2018},
    MRNUMBER = {3852659},
  publisher={JSTOR}
}

@article{bubeck2014theory,
	title={Theory of convex optimization for machine learning},
	author={Bubeck, S{\'e}bastien},
	journal={arXiv preprint arXiv:1405.4980},
	
	year={2014},

}

@book{vershynin2018high,
  title={High-dimensional probability: An introduction with applications in data science},
  author={Vershynin, Roman},
  volume={47},
  year={2018},
    MRNUMBER = {3837109},
  publisher={Cambridge university press}
}

@article{fan2018lamm,
  title={I-LAMM for sparse learning: Simultaneous control of algorithmic complexity and statistical error},
  author={Fan, Jianqing and Liu, Han and Sun, Qiang and Zhang, Tong},
  journal={Annals of statistics},
  volume={46},
  number={2},
  pages={814},
  year={2018},
    MRNUMBER = {3782385},
  publisher={NIH Public Access}
}

@article{brailovskaya2024universality,
  title={Universality and sharp matrix concentration inequalities},
  author={Brailovskaya, Tatiana and van Handel, Ramon},
  journal={Geometric and Functional Analysis},
  pages={1--105},
  year={2024},
    MRNUMBER = {4823211},
  publisher={Springer}
}

@article{dirksen2015tail,
  title={Tail bounds via generic chaining},
  author={Dirksen, Sjoerd},
  year={2015},
pages={1-29},
volume={20},
  MRNUMBER = {3354613},
journal={Electronic Journal of Probability}
}

@article{chen2020noisy,
  title={Noisy matrix completion: Understanding statistical guarantees for convex relaxation via nonconvex optimization},
  author={Chen, Yuxin and Chi, Yuejie and Fan, Jianqing and Ma, Cong and Yan, Yuling},
  journal={SIAM journal on optimization},
  volume={30},
  number={4},
  pages={3098--3121},
  year={2020},
    MRNUMBER = {4167625},
  publisher={SIAM}
}

@article{chen2015fast,
  title={Fast low-rank estimation by projected gradient descent: General statistical and algorithmic guarantees},
  author={Chen, Yudong and Wainwright, Martin J},
  journal={arXiv preprint arXiv:1509.03025},
  year={2015}
}

@article{keshavan2009matrix,
  title={Matrix completion from noisy entries},
  author={Keshavan, Raghunandan and Montanari, Andrea and Oh, Sewoong},
  journal={Advances in neural information processing systems},
  volume={22},
    MRNUMBER = {2678022},
  year={2009}
}

@article{candes2012exact,
  title={Exact matrix completion via convex optimization},
  author={Candes, Emmanuel and Recht, Benjamin},
  journal={Communications of the ACM},
  volume={55},
  number={6},
  pages={111--119},
  year={2012},
    MRNUMBER = {2565240},
  publisher={ACM New York, NY, USA}
}

@article{candes2010matrix,
  title={Matrix completion with noise},
  author={Candes, Emmanuel J and Plan, Yaniv},
  journal={Proceedings of the IEEE},
  volume={98},
  number={6},
  pages={925--936},
  year={2010},
  publisher={IEEE}
}

@article{gotze2021concentration,
  title={Concentration inequalities for polynomials in $\alpha$-sub-exponential random variables},
  author={G{\"o}tze, Friedrich and Sambale, Holger and Sinulis, Arthur},
  year={2021},
journal={Electronic Journal of Probability},
volume={26},
  MRNUMBER = {4247973},
pages={1-22}
}

@article{chen2024dynamic,
  title={Dynamic Matrix Recovery},
  author={Chen, Ziyuan and Yang, Ying and Yao, Fang},
  journal={Journal of the American Statistical Association},
  volume={119},
  number={548},
  pages={2996--3007},
  year={2024},
    MRNUMBER = {4833932},
  publisher={Taylor \& Francis}
}

@article{tropp2012user,
  title={User-friendly tail bounds for sums of random matrices},
  author={Tropp, Joel A},
  journal={Foundations of computational mathematics},
  volume={12},
  number={4},
  pages={389--434},
  year={2012},
    MRNUMBER = {2946459},
  publisher={Springer}
}

%% or include bibliography directly:
% \begin{thebibliography}{9}

% \bibitem{r1}
% \textsc{Billingsley, P.} (1999). \textit{Convergence of
% Probability Measures}, 2nd ed.
% Wiley, New York.
% \MR{1700749}

% \bibitem{r2}
% \textsc{Bourbaki, N.}  (1966). \textit{General Topology}  \textbf{1}.
% Addison--Wesley, Reading, MA.

% \bibitem{r3}
% \textsc{Ethier, S. N.} and \textsc{Kurtz, T. G.} (1985).
% \textit{Markov Processes: Characterization and Convergence}.
% Wiley, New York.
% \MR{838085}

% \bibitem{r4}
% \textsc{Prokhorov, Yu.} (1956).
% Convergence of random processes and limit theorems in probability
% theory. \textit{Theory  Probab.  Appl.}
% \textbf{1} 157--214.
% \MR{84896}
% \end{thebibliography}

\end{document}